%% file: arxiv2.tex
\documentclass{fancypaper}
\input{macros.tex}

\begin{document}

\author{H\`oa T. B\`ui}
\affil{%
  Curtin University
  \affilcr
  Centre for Optimisation and Decision Science
  \affilcr
  Kent Street, Bentley, Western Australia 6102, Australia
  \affilcr
  \email{hoa.bui@curtin.edu.au}
}

\author{Minh N. B\`ui}
\affil{%
  Chongqing Normal University
  \affilcr
  School of Mathematical Sciences
  \affilcr
  Shapingba District, Chongqing 401331, P.~R.~China
  \affilcr
  \email{mnb@cqnu.edu.cn}
}
\affil{%
  (previously) University of Graz
  \affilcr
  Department of Mathematics and Scientific Computing, NAWI Graz
  \affilcr
  Heinrichstra\ss{}e 36, 8010 Graz, Austria
}

\author{Christian Clason}
\affil{%
  University of Graz
  \affilcr
  Department of Mathematics and Scientific Computing, NAWI Graz
  \affilcr
  Heinrichstra{\ss}e 36, 8010 Graz, Austria
  \affilcr
  \email{c.clason@uni-graz.at}
}

\title{Variational Analysis in\\ Spectral Decomposition Systems\thanks{%
    Corresponding author: M.~N.~B\`ui.
    The work of H.~T.~B\`ui was supported by the Australian
    Research Council (grant number DE260100027).
    This research was funded in whole or in part by
    the Austrian Science Fund (FWF)
    \href{https://dx.doi.org/10.55776/F100800}{10.55776/F100800}.%
}}

\date{~}

\thispagestyle{plain.scrheadings}

\maketitle

\begin{abstract}
This work is concerned with the variational analysis of
functions defined on Euclidean spaces whose values depend solely
on certain invariants (``spectrum'') of their arguments,
a class we term ``spectral functions.''
Building on our previous work \cite{PartI} on the convex analysis of
such functions, we work in the abstract framework of spectral
decomposition systems, which covers a wide range of previously
studied settings, including eigenvalue decomposition of Hermitian
matrices and singular value decomposition of rectangular matrices,
and allows the derivation of new results in more general settings
such as normal decomposition systems.
The main results of this work provide constructive formulae for computing
the regular, limiting, and Clarke subdifferentials of a spectral
function in terms of the corresponding objects of the associated
invariant function.
Finally, we obtain a generalization of Lidski\u{\i}'s theorem
on the spectrum of additive perturbations of Hermitian matrices
to arbitrary spectral decomposition systems.
\end{abstract}


\newpage

\section{Introduction}

Many practically relevant optimization problems in Euclidean spaces
are naturally posed in terms of matrices or operators rather than
vectors. Notable examples include
phase retrieval \cite{Bauschke02,Luke02},
robust matrix estimation \cite{Benfenati20},
image recovery \cite{Chen19},
conic programming in Euclidean Jordan algebras \cite{Coey23},
low-rank approximation \cite{EckartYoung36,Markovsky08},
and nonlinear elasticity \cite{Rosakis98}.
In these applications, a frequently encountered class of functions
consists of those whose values depend
solely on certain invariants (\emph{spectrum}) of their arguments, which we
shall term \emph{spectral functions}.
Moreover, a large number of these functions are nonconvex,
including indicator functions of magnitude-based constraints in
phase retrieval \cite{Bauschke02,Luke02} and of
rank constraints \cite{DingSunYe14},
as well as Schatten $p$-norms for $0<p<1$ \cite{Marjanovic12}.

Structurally, such a spectral function $\Phi$ can be written
as the composition of a ``simpler'' \emph{associated invariant
function} $\varphi$ defined on another Euclidean space and a
\emph{spectral mapping} $\gamma$.
Thus, a central theme in the variational analysis of spectral
functions is the computation of fundamental objects such as
regular, limiting, and Clarke
subdifferentials of $\Phi$ through the corresponding objects of the
associated invariant function $\varphi$.
In the existing literature
\cite{Baes07,Drus18,Lewis96c,Lewis99b,Lewis05a,Lewis05b,Lourenco20,Sendov07,Sun08},
this is commonly accomplished through the following
reduction-and-lifting scheme:
\begin{enumerate}[label={\normalfont\arabic*.}]
\item
Compute the corresponding object of the simpler associated invariant
function $\varphi$ at $\gamma(X)$.
\item
Lift the result back to the original space through a suitable
embedding operator to obtain the desired result.
\end{enumerate}
Despite the strikingly similar form of the resulting formulae
in these works, their proofs rely on techniques that are intrinsic
to the underlying framework. To the best of our knowledge,
these subdifferential results have thus far not been brought
together under a common abstract framework,
thereby obscuring the unifying principles underlying them and
necessitating separate techniques for each new ``spectral setting,''
including the normal decomposition system framework of
\cite{Lewis96b} for which no generalized-subdifferential formulae
for spectral functions appear to be
available.

It is therefore the goal of the present paper to address this gap.
To this end, we build upon the notion of a
\emph{spectral decomposition system} developed in our previous work
\cite{PartI}.
At the core of a spectral decomposition system are
a family of embedding operators together with an abstract
spectral decomposition property, both motivated by the various
decomposition results for matrices. These objects provide the
mechanism underlying the reduction-and-lifting scheme above.
This abstract framework encompasses
the aforementioned settings, namely
Euclidean Jordan algebras,
eigenvalues of Hermitian matrices,
singular values of rectangular matrices,
signed singular values of square matrices,
and normal decomposition systems.
In the present paper, we develop within
this spectral decomposition system framework
a unified approach for deriving
generalized-subdifferential formulae for spectral functions.
Our results simultaneously recover the existing
generalized-subdifferential formulae for
spectral functions of eigenvalues of
Hermitian matrices, singular values of rectangular matrices,
and eigenvalues in Euclidean Jordan algebras,
while also yielding new results for normal decomposition systems,
thereby settling the open question raised in the discussion in
the paragraph following \cite[Theorem~7.2]{Lewis03}.

Let us emphasize that our general approach
differs fundamentally from the existing literature. It allows us to
bypass the framework-specific proof techniques of existing
works via a geometric approach common in variational analysis:
First, we leverage the results in the first part \cite{PartI} to
establish characterizations of normal
cones to spectral sets, which are of interest in their own right.
We then transfer these results via normal cones to
epigraphs (with the help of the ``product-space spectral
decomposition system'' in \cref{p:10}) to the
subdifferentials of interest.

This work is organized as follows.
In \cref{sec:2}, we recall the precise definition of
spectral decomposition systems, state our standing assumptions, and
illustrate how this framework encompasses previously studied
settings.
We also recall from \cite{PartI} various basic properties that are
fundamental to the analysis in this work.
\Cref{sec:3} is then concerned with the variational geometry of
spectral sets, where we give explicit formulae for the regular and
limiting normal cones (\cref{p:2}).
In \cref{sec:4}, we leverage these results to obtain
formulae for the regular and limiting subdifferentials
of a spectral function in terms of those of the associated
invariant function (\cref{t:1}). The former in particular yields a
full characterization of Fréchet differentiability and of Fréchet
derivatives (\cref{c:5}) of spectral functions.
We also derive a representation of the Clarke subdifferential of
spectral functions (\cref{p:3}), which is then employed in the
final \cref{sec:5} to derive a generalization of Lidski\u{\i}'s
theorem in spectral decomposition systems (\cref{t:3}).

\section{Spectral decomposition systems}
\label{sec:2}

Let us first introduce several essential notions and terminology.
Let $\HH$ be a Euclidean space, that is, a finite-dimensional real
Hilbert space. The scalar product and its associated norm are
denoted by $\scal{\Cdot}{\Cdot}$ and $\norm{\Cdot}$, respectively.
(We will use $\norm{\Cdot}_{\HH}$ when there is a potential
ambiguity.)
The space of linear operators from $\HH$ to a Euclidean space
$\GG$ is denoted by $\mathscr{L}(\HH,\GG)$ and we equip it with the
topology induced by the norm
\begin{equation}
  \brk!{\forall L\in\mathscr{L}(\HH,\GG)}\quad
  \norm{L}_{\mathscr{L}(\HH,\GG)}
  =\max_{\substack{x\in\HH\\ \norm{x}\leq 1}}\norm{Lx},
\end{equation}
and we write $\mathscr{L}(\HH)=\mathscr{L}(\HH,\HH)$.
(Since we are working in finite-dimensional spaces, this topology
coincides with the topology of pointwise convergence.)
Now let $\SD$ be a nonempty set which acts on $\HH$ by linear
isometries in the sense that $\HH$ is endowed with a mapping
$\SD\times\HH\to\HH\colon(\sd,x)\mapsto\sd\cdot x$
such that, for every $\sd\in\SD$,
$\HH\to\HH\colon x\mapsto\sd\cdot x$ is a linear isometry.
Then:
\begin{itemize}
\item
Given a set $\UU$, a mapping $\Phi\colon\HH\to\UU$ is said to be
\emph{$\SD$-invariant} if
\begin{equation}
(\forall x\in\HH)(\forall\sd\in\SD)\quad
\Phi(\sd\cdot x)=\Phi(x).
\end{equation}
\item
A subset $C$ of $\HH$ is said to be
\emph{$\SD$-invariant} if its indicator function
\begin{equation}
\iota_C\colon\HH\to\RX\colon x\mapsto
\begin{cases}
0&\text{if}\,\,x\in C,\\
\pinf&\text{otherwise},
\end{cases}
\end{equation}
is $\SD$-invariant or, equivalently,
\begin{equation}
(\forall x\in\HH)(\forall\sd\in\SD)\quad
x\in C
\quad\Leftrightarrow\quad
\sd\cdot x\in C.
\end{equation}
\end{itemize}

\subsection{Definitions, assumptions, and characterization}

We are now in a position to present the notion of a spectral
decomposition system introduced in \cite{PartI}.
The present work requires only the following special case of the
general framework developed there.

\begin{definition}[spectral decomposition system]
\label{d:1}
Let $\FH$ and $\XX$ be Euclidean spaces,
let $\SD$ be a nonempty set which acts on $\XX$ by linear
isometries, let $\gamma\colon\FH\to\XX$, and suppose that
$(\Lambda_{\ad})_{\ad\in\AD}$ is a family of linear
isometries from $\XX$ to $\FH$.
The tuple
$\mathfrak{S}=(\XX,\SD,\gamma,(\Lambda_{\ad})_{\ad\in\AD})$
is said to be a \emph{spectral decomposition system} for $\FH$ if
the following are satisfied:
\begin{enumerate}[label={\normalfont[\Alph*]}]
\item
\label{d:1a}
There exists an $\SD$-invariant mapping $\tau\colon\XX\to\XX$ such
that
\begin{equation}
\label{e:da2t}
\brk[s]!{\,
(\forall x\in\XX)(\exi\sd\in\SD)\,\,
x=\sd\cdot\tau(x)
\,}
\quad\text{and}\quad
\brk[s]!{\,
(\forall\ad\in\AD)\,\,
\gamma\circ\Lambda_{\ad}=\tau
\,}.
\end{equation}
\item
\label{d:1b}
$(\forall X\in\FH)(\exi\ad\in\AD)$
$X=\Lambda_{\ad}\gamma(X)$.
\item
\label{d:1c}
$(\forall X\in\FH)(\forall Y\in\FH)$
$\scal{X}{Y}\leq\scal{\gamma(X)}{\gamma(Y)}$.
\end{enumerate}
If $\mathfrak{S}=(\XX,\SD,\gamma,(\Lambda_{\ad})_{\ad\in\AD})$
is a spectral decomposition system for $\FH$, then:
\begin{itemize}
\item
The mapping $\gamma$ is called the \emph{spectral mapping} of the system
$\mathfrak{S}$.
\item
The mapping $\tau$ in property~\cref{d:1a} is called the
\emph{spectral-induced ordering mapping} of the system $\mathfrak{S}$.
(The motivation for this term will become apparent
from the concrete scenarios described in
\cref{ex:3,ex:4,ex:5,ex:6}.)
\item
We set
\begin{equation}
\label{e:ax}
(\forall X\in\FH)\quad
\AD_X=\menge{\ad\in\AD}{X=\Lambda_{\ad}\gamma(X)}.
\end{equation}
Property~\cref{d:1b} ensures that the sets $(\AD_X)_{X\in\FH}$
are nonempty.
\item
Given $X\in\FH$, the vector $\gamma(X)$ is called the \emph{spectrum} of
$X$ with respect to $\mathfrak{S}$ and, for every $\ad\in\AD_X$,
the identity
$X=\Lambda_{\ad}\gamma(X)$
is called a \emph{spectral decomposition} of $X$ with respect to
$\mathfrak{S}$.
\end{itemize}

\end{definition}

Our results will be formulated under the following assumption,
which requires a closedness condition on the set
$\set{\Lambda_{\ad}}_{\ad\in\AD}$ of embedding operators
in \cref{d:1} that
will be essential in several limiting arguments in
this paper. As we shall demonstrate in \cref{sec:22}, this
assumption is satisfied in all previously considered examples.

\begin{assumption}
  \label{a:1}
  $\FH$ is a Euclidean space and
  $\mathfrak{S}=(\XX,\SD,\gamma,(\Lambda_{\ad})_{\ad\in\AD})$ is a
  spectral decomposition system for $\FH$. Moreover, the set
  $\set{\Lambda_{\ad}}_{\ad\in\AD}$ is closed in $\mathscr{L}(\XX,\FH)$.
\end{assumption}

Next, we formally define the classes of functions and sets
introduced in the Introduction, namely those whose values or
membership depend solely on the spectrum of their arguments.
As will be seen in \cref{sec:22}, this definition captures various
existing classes of spectral functions.
Here we use the notation $\RXX=\intv{\minf}{\pinf}$.

\begin{definition}[spectral function and spectral set]
\label{d:2}
Let $\FH$ be a Euclidean space and let
$(\XX,\SD,\gamma,(\Lambda_{\ad})_{\ad\in\AD})$ be a
spectral decomposition system for $\FH$.
\begin{enumerate}
\item\label{d:2i}
A function $\Phi\colon\FH\to\RXX$ is said to be a
\emph{spectral function} if
\begin{equation}
(\forall X\in\FH)(\forall Y\in\FH)\quad
\gamma(X)=\gamma(Y)
\quad\Rightarrow\quad
\Phi(X)=\Phi(Y).
\end{equation}
\item\label{d:2ii}
A set $\DD\subset\FH$ is said to be a \emph{spectral set} if its
indicator function $\iota_{\DD}$ is a spectral function or,
equivalently,
\begin{equation}
(\forall X\in\FH)(\forall Y\in\FH)\quad
\begin{cases}
\gamma(X)=\gamma(Y)\\
X\in\DD
\end{cases}
\quad\Rightarrow\quad
Y\in\DD.
\end{equation}
\end{enumerate}
\end{definition}

The next result characterizes spectral functions precisely
as those of the form
\begin{equation}
  \Phi=\varphi\circ\gamma,
  \quad\text{where $\varphi\colon\XX\to\RXX$ is
    $\SD$-invariant},
\end{equation}
which thus motivates the study of
variational-analytic properties and objects associated with
$\Phi$ in terms of those of $\varphi$.

\begin{proposition}[\protect{\cite[Proposition~2.4]{PartI}}]
\label{p:85}
Let $\FH$ be a Euclidean space, and let
$(\XX,\SD,\gamma,(\Lambda_{\ad})_{\ad\in\AD})$ be a
spectral decomposition system for $\FH$.
Then a function $\Phi\colon\FH\to\RXX$ is a spectral function
if and only if there exists an $\SD$-invariant function
$\varphi\colon\XX\to\RXX$ such that
$\Phi=\varphi\circ\gamma$,
in which case, $\varphi$ is uniquely determined by
\begin{equation}
(\forall\ad\in\AD)\quad\varphi=\Phi\circ\Lambda_{\ad},
\end{equation}
and we call $\varphi$ the
\emph{invariant function associated with $\Phi$}.
\end{proposition}
We immediately obtain a characterization of spectral sets precisely
as those of the form
\begin{equation}
  \DD=\gamma^{-1}(D),
  \quad\text{where $D\subset\XX$ is $\SD$-invariant}
\end{equation}
by applying \cref{p:85} to $\Phi=\iota_{\DD}$.
\begin{corollary}
\label{c:1}
Let $\FH$ be a Euclidean space, and let
$(\XX,\SD,\gamma,(\Lambda_{\ad})_{\ad\in\AD})$ be a
spectral decomposition system for $\FH$.
Then a set $\DD\subset\FH$ is a spectral set if and only if
there exists an $\SD$-invariant set $D\subset\XX$
such that
$\DD=\gamma^{-1}(D)$,
in which case $D$ is uniquely determined by
\begin{equation}
(\forall\ad\in\AD)\quad
D=\Lambda_{\ad}^{-1}(\DD),
\end{equation}
and we call $D$ the
\emph{invariant set associated with $\DD$}.
\end{corollary}

\subsection{Examples}
\label{sec:22}

We now illustrate the scope of \cref{a:1} and the notion of a
spectral function.
The examples here are mostly drawn from our earlier work
\cite{PartI}, where we first introduced this notion and to which we
refer the reader for additional background and details.
Furthermore, we verify that the additional closedness assumption
imposed in this work is satisfied in each of these examples. To
that end, let us introduce some frequently used notation.

\begin{notation}
\label{n:1}
Let $M$ and $N$ be strictly positive integers.
\begin{itemize}
\item
$\KB$ denotes one of the following:
the field $\RR$ of real numbers,
the field $\CC$ of complex numbers,
or the skew-field $\HB$ of Hamiltonian quaternions
(we refer the reader to \cite{Rodman14} for background on
quaternions).
\item
The canonical involution on $\KB$ is denoted by
$\xi\mapsto\overline{\xi}$, whose set of fixed points is
$\RR$, and the real part of $\xi\in\KB$ is
$\Re\xi=(\xi+\overline{\xi})/2$.
\item
$\KB^{M\times N}$ denotes the real vector space of $M\times N$
matrices with entries in $\KB$.
\item
The conjugate transpose of
$X=\brk[s]{\xi_{ij}}_{\substack{1\leq i\leq M\\ 1\leq j\leq N}}
\in\KB^{M\times N}$ is
$X^*
=\brk[s]{\overline{\xi_{ji}}}_{\substack{1\leq j\leq N\\ 1\leq i\leq M}}
\in\KB^{N\times M}$.
\item
The trace of a matrix $X\in\KB^{N\times N}$ is denoted by $\tra X$.
\item
$\HS^N(\KB)=\menge{X\in\KB^{N\times N}}{X=X^*}$
is the vector subspace of $\KB^{N\times N}$ which consists of
Hermitian matrices.
\item
$\US^N(\KB)=\menge{U\in\KB^{N\times N}}{U^*U=\Id}$ is the
set of unitary matrices.
\item
$\SO^N=\menge{U\in\US^N(\RR)}{\det U=1}$ is the special orthogonal
group.
\item
$\PS_{\pm}^N$ denotes the set of $N\times N$ \emph{signed
permutation matrices}, that is, matrices which contain exactly one
nonzero entry in every row and every column, and that entry is $-1$
or $1$.
\item
$\PS^N$ is the set of \emph{permutation matrices}, that is, the
subset of $\PS_{\pm}^N$ which consists of matrices with entries in
$\set{0,1}$.
\end{itemize}
\end{notation}

The first example shows that every Euclidean space admits a spectral
decomposition system satisfying \cref{a:1}.

\begin{example}
\label{ex:1}
Let $\HH$ be a nonzero Euclidean space,
let the set $\set{-1,1}$ act on $\RR$ via multiplication, denote by
$S$ the unit sphere of $\HH$, and set
\begin{equation}
(\forall e\in S)\quad
\Lambda_e\colon\RR\to\HH\colon\xi\mapsto\xi e.
\end{equation}
Then $(\RR,\set{-1,1},\norm{\Cdot},(\Lambda_e)_{e\in S})$
is a spectral decomposition system for $\HH$,
and the set $\set{\Lambda_e}_{e\in S}$ is closed in
$\mathscr{L}(\RR,\HH)$.
In this setting, spectral functions are precisely
\emph{radial} functions, that is, functions
$\Phi\colon\HH\to\RXX$ that satisfy
\begin{equation}
(\forall x\in\HH)(\forall y\in\HH)\quad
\norm{x}=\norm{y}
\quad\Rightarrow\quad
\Phi(x)=\Phi(y).
\end{equation}
\end{example}
\begin{proof}
The assertion that
$(\RR,\set{-1,1},\norm{\Cdot},(\Lambda_e)_{e\in S})$
is a spectral decomposition system for $\HH$
is proved in \cite[Example~2.7]{PartI},
while the closedness of $\set{\Lambda_e}_{e\in S}$ in
$\mathscr{L}(\RR,\HH)$ follows from the compactness
of $S$ in $\HH$.
\end{proof}

We now show that spectral decomposition systems
encompass the normal decomposition system framework of
\cite{Lewis96b}, which in turn subsumes the Lie-theoretic
framework of \cite{Berezin56,Lewis00,Tam98}.

\begin{example}[normal decomposition system]
\label{ex:2}
Let $(\HH,\mathscr{G},\gamma)$ be a \emph{normal decomposition
system} in the sense of \cite[Definition~2.1]{Lewis96b}, that is,
$\HH$ is a Euclidean space,
$\mathscr{G}$ is a closed subgroup of the group
$\mathscr{O}(\HH)$ (equipped with mapping composition and induced
topology from $\mathscr{L}(\HH)$) of all orthogonal
linear operators on $\HH$, and $\gamma\colon\HH\to\HH$ is a
$\mathscr{G}$-invariant mapping such that
\begin{equation}
\begin{cases}
(\forall x\in\HH)(\exi T\in\mathscr{G})\,\,
x=T\gamma(x)\\
(\forall x\in\HH)(\forall y\in\HH)\,\,
\scal{x}{y}\leq\scal{\gamma(x)}{\gamma(y)}.
\end{cases}
\end{equation}
Further, let $\XX$ be a vector subspace of $\HH$ that contains the
range of $\gamma$, set
\begin{equation}
\mathscr{G}_{\XX}=\menge{T\in\mathscr{G}}{T(\XX)=\XX},
\end{equation}
and let $\mathscr{G}_{\XX}$ act on $\XX$ via
$(T,x)\mapsto Tx$.
Suppose that
\begin{equation}
(\forall x\in\XX)(\exi T\in\mathscr{G}_{\XX})\quad
x=T\gamma(x).
\end{equation}
(See \cite[Assumption~4.1]{Lewis96b}.)
Then $\mathfrak{S}=(\XX,\mathscr{G}_{\XX},\gamma,
(\restr{T}{\XX})_{T\in\mathscr{G}})$
is a spectral decomposition system for $\HH$,
and $\set{\restr{T}{\XX}}_{T\in\mathscr{G}}$ is closed in
$\mathscr{L}(\XX,\HH)$.
In this setting, spectral functions are precisely
$\mathscr{G}$-invariant functions $\Phi\colon\HH\to\RXX$.
\end{example}
\begin{proof}
As shown in \cite[Example~2.11]{PartI}, $\mathfrak{S}$
is a spectral decomposition system for $\HH$.
Next, let us verify that $
\mathscr{K}=\set{\restr{T}{\XX}}_{T\in\mathscr{G}}$ is
closed in $\mathscr{L}(\XX,\HH)$.
Let $(\restr{T_n}{\XX})_{n\in\NN}$ be a sequence in $\mathscr{K}$
that converges in $\mathscr{L}(\XX,\HH)$ to some operator $L$.
Since $\mathscr{G}$ is closed in $\mathscr{O}(\HH)$
and $\mathscr{O}(\HH)$ is compact, it follows that $\mathscr{G}$ is
compact in $\mathscr{O}(\HH)$. Thus,
there exist a strictly increasing sequence
$(k_n)_{n\in\NN}$ in $\NN$ and an operator
$T\in\mathscr{G}$ such that
$(\forall x\in\HH)$ $T_{k_n}x\to Tx$.
Thus, since $\restr{T_{k_n}}{\XX}\to L$ in
$\mathscr{L}(\XX,\HH)$, we deduce that
$(\forall x\in\XX)$ $Lx=\lim T_{k_n}x=Tx$.
Therefore $L=\restr{T}{\XX}\in\mathscr{K}$.
\end{proof}

The next illustration is the Euclidean Jordan algebra
framework of \cite{Baes07,Lourenco20,Sendov07,Sun08},
which captures the space $\HS^N(\KB)$ of Hermitian matrices (see
\cref{ex:4}).
It was shown in \cite{Orlitzky22} that, in general, Euclidean Jordan
algebras cannot be embedded into a normal decomposition system of
\cref{ex:2}, demonstrating that these are distinct notions.

\begin{example}[Euclidean Jordan algebra]
  \label{ex:3}
  Let $\FH$ be a \emph{Euclidean Jordan algebra}, that is, $\FH$ is a
  finite-dimensional real vector
  space which is endowed with a bilinear form
  \begin{equation}
    \label{e:bform}
    \FH\times\FH\to\FH\colon
    (X,Y)\mapsto X\circledast Y
  \end{equation}
  such that the following are satisfied:
  \begin{enumerate}[label={\normalfont[\Alph*]}]
    \item
      $(\forall X\in\FH)(\forall Y\in\FH)$
      $X\circledast Y=Y\circledast X$ and
      $X\circledast((X\circledast X)\circledast Y)
      =(X\circledast X)\circledast(X\circledast Y)$.
    \item
      There exists a scalar product
      $\killing{\Cdot}{\Cdot}$ on $\FH$ such that
      $(\forall X\in\FH)(\forall Y\in\FH)(\forall Z\in\FH)$
      $\killing{X\circledast Y}{Z}=\killing{X}{Y\circledast Z}$;
  \end{enumerate}
  see \cite{Faraut94} for background and complements on
  Euclidean Jordan algebras.
  We equip $\FH$ with the trace scalar product
  \begin{equation}
    \label{e:wcyk}
    (\forall X\in\FH)(\forall Y\in\FH)\quad
    \scal{X}{Y}=\Tra(X\circledast Y),
  \end{equation}
  where $\Tra X$ denotes the trace in $\FH$ of an element $X\in\FH$
  (see \cite[Section~II.2]{Faraut94}).
  Denote by $E$ the identity element of $\FH$ and by $N$ the
  rank of $\FH$, and let $\PS^N$ act on $\RR^N$ via matrix-vector
  multiplication. Following \cite{Baes07}, we define a \emph{Jordan
  frame} of $\FH$ as a family $(A_i)_{1\leq i\leq N}$ in
  $\FH^N\smallsetminus\set{0}$ such that
  \begin{equation}
    \label{e:y7by}
    \begin{cases}
      \brk!{\forall i\in\set{1,\ldots,N}}
      \brk!{\forall j\in\set{1,\ldots,N}}\,\,
      A_i\circledast A_j=
      \begin{cases}
        A_i&\text{if}\,\,i=j,\\
        0&\text{if}\,\,i\neq j,
      \end{cases}\\
      \Sum_{i=1}^NA_i=E.
    \end{cases}
  \end{equation}
  (This is equivalent to the standard definition, e.g., on
  \cite[p.~44]{Faraut94}, of a complete system of orthogonal
  primitive idempotents satisfying \eqref{e:y7by} since such a
  system is necessarily of size $N$ and vice versa.)
  Next, the spectral decomposition theorem for Euclidean Jordan
  algebras \cite[Theorem~III.1.2]{Faraut94} states that,
  for every $X\in\FH$, there exist a
  unique vector $(\lambda_1(X),\ldots,\lambda_N(X))\in\RR^N$, the
  entries of which are called the \emph{eigenvalues} of $X$, 
  and a Jordan frame $(A_i)_{1\leq i\leq N}$ such that
  \begin{equation}
    \lambda_1(X)\geq\cdots\geq\lambda_N(X)
    \quad\text{and}\quad
    X=\sum_{i=1}^N\lambda_i(X)A_i.
  \end{equation}
  We thus obtain a mapping
  \begin{equation}
    \label{e:9vmn}
    \lambda\colon\FH\to\RR^N\colon
    X\mapsto\brk!{\lambda_1(X),\ldots,\lambda_N(X)}.
  \end{equation}
  Further, denote by $\AD$ the set of Jordan frames of $\FH$ and
  define
  \begin{equation}
    \label{e:cwpw}
    \brk!{\forall\ad=(A_i)_{1\leq i\leq N}\in\AD}\quad
    \Lambda_{\ad}\colon\RR^N\to\FH\colon
    x=(\xi_i)_{1\leq i\leq N}\mapsto\sum_{i=1}^N\xi_iA_i.
  \end{equation}
  Then $\mathfrak{S}=(\RR^N,\PS^N,\lambda,(\Lambda_{\ad})_{\ad\in\AD})$
  is a spectral decomposition system for $\FH$,
  and the set $\set{\Lambda_{\ad}}_{\ad\in\AD}$ is closed in
  $\mathscr{L}(\RR^N,\FH)$.
  In this setting, the class of spectral functions coincides with
  that studied in \cite{Baes07,Lourenco20,Sendov07,Sun08}.
\end{example}
\begin{proof}
The first assertion is proved in \cite[Example~2.5]{PartI}.
To establish the closedness of
$\set{\Lambda_{\ad}}_{\ad\in\AD}$ in $\mathscr{L}(\RR^N,\FH)$,
let $(\ad_n)_{n\in\NN}$ be a sequence of Jordan frames
of $\FH$ such that $(\Lambda_{\ad_n})_{n\in\NN}$ converges
in $\mathscr{L}(\RR^N,\FH)$ to some operator $L$.
We must show that there exists a Jordan frame $\ad\in\AD$ such that
$L=\Lambda_{\ad}$.
To do so, denote by $(e_i)_{1\leq i\leq N}$ the canonical basis of
$\RR^N$, set
\begin{equation}
\brk!{\forall i\in\set{1,\ldots,N}}\quad
A_i=Le_i,
\end{equation}
and, for every $n\in\NN$, we write
$\ad_n=(A_{i,n})_{1\leq i\leq N}$.
We deduce from \cref{e:cwpw} that
\begin{equation}
\brk!{\forall i\in\set{1,\ldots,N}}\quad
A_{i,n}
=\Lambda_{\ad_n}e_i
\to Le_i
=A_i.
\end{equation}
On the other hand, for every $n\in\NN$, since
$(A_{i,n})_{1\leq i\leq N}$ is a Jordan frame of $\FH$, we have
\begin{equation}
\begin{cases}
\brk!{\forall i\in\set{1,\ldots,N}}
\brk!{\forall j\in\set{1,\ldots,N}}\,\,
A_{i,n}\circledast A_{j,n}=
\begin{cases}
A_{i,n}&\text{if}\,\,i=j,\\
0&\text{if}\,\,i\neq j
\end{cases}\\
\brk!{\forall i\in\set{1,\ldots,N}}\,\,
\norm{A_{i,n}}=1\\
\Sum_{i=1}^NA_{i,n}=E.
\end{cases}
\end{equation}
Hence, since the bilinear form \cref{e:bform} is continuous (note
that $\FH$ is finite-dimensional), letting $n\to\pinf$ shows that
$(A_i)_{1\leq i\leq N}$ is a family in $\FH\smallsetminus\set{0}$
that satisfies \cref{e:y7by}.
Therefore, $(A_i)_{1\leq i\leq N}$ is a Jordan frame of $\FH$.
Furthermore, we derive from the linearity of $L$ that
\begin{equation}
\brk!{\forall x=(\xi_i)_{1\leq i\leq N}\in\RR^N}\quad
Lx
=L\brk3{\sum_{i=1}^N\xi_ie_i}
=\sum_{i=1}^N\xi_iLe_i
=\sum_{i=1}^N\xi_iA_i.
\end{equation}
Consequently, upon setting $\ad=(A_i)_{1\leq i\leq N}\in\AD$, we
conclude that $L=\Lambda_{\ad}$.
\end{proof}

Specializing \cref{ex:3} to the Euclidean Jordan algebra of
Hermitian matrices (see \cite[Section~V.2]{Faraut94}),
we obtain the following framework.

\begin{example}[eigenvalue decomposition]
\label{ex:4}
Let $2\leq N\in\NN$. We equip $\HS^N(\KB)$ with the trace scalar
product
$\scal{\Cdot}{\Cdot}\colon
(X,Y)\mapsto\Re\tra(XY)$
and let $\PS^N$ act on $\RR^N$ via matrix-vector multiplication.
For every $X\in\HS^N(\KB)$, let
$\lambda(X)=(\lambda_1(X),\ldots,\lambda_N(X))$ be the vector
formed by the $N$ eigenvalues (counting multiplicities) of $X$
listed in decreasing order (see \cite[Theorem~5.3.6\,(c)]{Rodman14}
for the quaternion case). Moreover, set
\begin{equation}
\brk!{\forall U\in\US^N(\KB)}\quad
\Lambda_U\colon\RR^N\to\HS^N(\KB)\colon
x\mapsto U(\Diag x)U^*.
\end{equation}
Then
$\mathfrak{S}=(\RR^N,\PS^N,\lambda,(\Lambda_U)_{U\in\US^N(\KB)})$
is a spectral decomposition system for $\HS^N(\KB)$,
and the set $\set{\Lambda_U}_{U\in\US^N(\KB)}$ is closed in
$\mathscr{L}(\RR^N,\HS^N(\KB))$.
In this setting, spectral functions are precisely
\emph{weakly unitarily invariant} functions,
that is, functions $\Phi\colon\HS^N(\KB)\to\RXX$ that satisfy
\begin{equation}
\brk!{\forall X\in\HS^N(\KB)}
\brk!{\forall U\in\US^N(\KB)}\quad
\Phi(UXU^*)=\Phi(X).
\end{equation}
\end{example}

The next example concerns the singular value decomposition of
rectangular matrices.

\begin{example}[singular value decomposition]
\label{ex:5}
Let $M$ and $N$ be strictly positive integers
and set $m=\min\set{M,N}$.
Let $\FH$ be the Euclidean space obtained by equipping
$\KB^{M\times N}$ with the trace scalar product
$(X,Y)\mapsto\Re\tra(X^*Y)$,
and let $\PS_{\pm}^m$ act on $\RR^m$ via matrix-vector
multiplication. Given a matrix $X\in\FH$, the vector in
$\RP^m$ formed by the $m$ singular values (counting multiplicities)
of $X$, with the convention that they are listed in decreasing
order, is denoted by $(\sigma_1(X),\ldots,\sigma_m(X))$; see
\cite[Proposition~3.2.5\,(f)]{Rodman14} for
singular value decomposition of matrices in $\HB^{M\times N}$.
This thus defines a mapping
\begin{equation}
\sigma\colon\FH\to\RR^m\colon
X\mapsto\brk!{\sigma_1(X),\ldots,\sigma_m(X)}.
\end{equation}
Further, set $\AD=\US^M(\KB)\times\US^N(\KB)$ and
\begin{equation}
\label{e:4l5f}
\brk!{\forall\ad=(U,V)\in\AD}\quad
\Lambda_{\ad}\colon\RR^m\to\FH\colon
x\mapsto U(\Diag x)V^*,
\end{equation}
where the operator $\Diag\colon\RR^m\to\FH$ maps a vector
$(\xi_i)_{1\leq i\leq m}$ to the diagonal matrix in $\FH$ of which
the $i$th diagonal entry is $\xi_i$. Then
$\mathfrak{S}=(\RR^m,\PS_{\pm}^m,\sigma,(\Lambda_{\ad})_{\ad\in\AD})$
is a spectral decomposition system for $\FH$,
and the set $\set{\Lambda_{\ad}}_{\ad\in\AD}$ is closed in
$\mathscr{L}(\RR^m,\FH)$.
In this setting, spectral functions
are precisely \emph{unitarily invariant} functions, that is,
functions $\Phi\colon\FH\to\RXX$ that satisfy
\begin{equation}
(\forall X\in\FH)\brk!{\forall U\in\US^M(\KB)}
\brk!{\forall V\in\US^N(\KB)}\quad
\Phi(UXV^*)=\Phi(X).
\end{equation}
\end{example}
\begin{proof}
It was established in \cite[Example~2.14]{PartI}
that $\mathfrak{S}$ is a spectral decomposition system for $\FH$,
while the closedness of $\set{\Lambda_{\ad}}_{\ad\in\AD}$ in
$\mathscr{L}(\RR^m,\FH)$ follows from the compactness of
$\AD$ in $\KB^{M\times M}\times\KB^{N\times N}$ as well as the
continuity of the mapping $\AD\to\mathscr{L}(\RR^m,\FH)\colon
\ad\mapsto\Lambda_{\ad}$.
\end{proof}

Our last example revolves around signed singular values of square
matrices, a framework that arose in the study of
isotropic stored energy functions in nonlinear elasticity
\cite{Rosakis98} and in the problem of characterizing matrices with
prescribed singular values and main diagonal elements
\cite{Thompson77}. Several convex analysis results in this setting
were established in \cite{Dacorogna07}.

\begin{example}[signed singular value decomposition]
\label{ex:6}
Let $2\leq N\in\NN$ and let $\FH$ be the Euclidean space obtained by
equipping $\RR^{N\times N}$ with the trace scalar product
$(X,Y)\mapsto\tra(X\tran Y)$,
let $\SD$ be the subset of $\PS_{\pm}^N$ which consists of all
matrices with an even number of entries equal to $-1$, and let
$\SD$ act on $\RR^N$ via matrix-vector multiplication. As in
\cref{ex:5}, $\sigma(X)=(\sigma_1(X),\ldots,\sigma_N(X))$
designates the vector formed by the $N$ singular values (counting
multiplicities) of a matrix $X\in\FH$, with the convention that
$\sigma_1(X)\geq\cdots\geq\sigma_N(X)$. Define a mapping
\begin{equation}
\gamma\colon\FH\to\RR^N\colon
X\mapsto\brk!{\gamma_1(X),\ldots,\gamma_N(X)}
\end{equation}
by
\begin{equation}
(\forall X\in\FH)
\brk!{\forall i\in\set{1,\ldots,N}}\quad
\gamma_i(X)=
\begin{cases}
\sigma_i(X)&\text{if}\,\,1\leq i\leq N-1,\\
\sigma_N(X)\sign(\det X)&\text{if}\,\,i=N.
\end{cases}
\end{equation}
Finally, set $\AD=\SO^N\times\SO^N$ and
\begin{equation}
\brk!{\forall\ad=(U,V)\in\AD}\quad
\Lambda_{\ad}\colon\RR^N\to\FH\colon
x\mapsto U(\Diag x)V\tran.
\end{equation}
Then $\mathfrak{S}=(\RR^N,\SD,\gamma,(\Lambda_{\ad})_{\ad\in\AD})$
is a spectral decomposition system for $\FH$,
and the set $\set{\Lambda_{\ad}}_{\ad\in\AD}$
is closed in $\mathscr{L}(\RR^N,\FH)$.
In this setting, spectral functions are precisely
\emph{$\SO^N\times\SO^N$-invariant} functions, that is, functions
$\Phi\colon\FH\to\RXX$ that satisfy
\begin{equation}
(\forall X\in\FH)
\brk!{\forall U\in\SO^N}
\brk!{\forall V\in\SO^N}\quad
\Phi\brk!{UXV\tran}=\Phi(X).
\end{equation}
\end{example}
\begin{proof}
It was shown in \cite[Example~2.15]{PartI} that
$\mathfrak{S}$ is a spectral decomposition system for $\FH$.
Furthermore, the closedness of
$\set{\Lambda_{\ad}}_{\ad\in\AD}$ follows from the compactness of
$\SO^N$ in $\FH$ and the continuity of the mapping
$\AD\to\mathscr{L}(\RR^N,\FH)\colon\ad\mapsto\Lambda_{\ad}$.
\end{proof}

\subsection{Fundamental properties}

In this subsection, we gather several essential properties of
spectral decomposition systems that will be used throughout the
paper.

\begin{proposition}
\label{p:7}
Suppose that \cref{a:1} is in force,
let $x$ and $y$ be in $\XX$, and let $\sd\in\SD$.
Then $\scal{\sd\cdot x}{\sd\cdot y}=\scal{x}{y}$.
\end{proposition}
\begin{proof}
Since
$\XX\to\XX\colon z\mapsto\sd\cdot z$ is a linear isometry, we
obtain
$2\scal{\sd\cdot x}{\sd\cdot y}
=\norm{\sd\cdot x}^2+\norm{\sd\cdot y}^2
-\norm{\sd\cdot x-\sd\cdot y}^2
=\norm{x}^2+\norm{y}^2-\norm{x-y}^2
=2\scal{x}{y}$,
as claimed.
\end{proof}

The following properties are collected from \cite[Proposition~3.5\,(iv)]{PartI} and
\cite[Proposition~3.8\,(i)]{PartI}, respectively.
\begin{proposition}
\label{p:8}
Suppose that \cref{a:1} is in force. Then the following hold:
\begin{enumerate}
\item
\label{p:8i}
$\gamma$ is nonexpansive, that is,
$(\forall X\in\FH)(\forall Y\in\FH)$
$\norm{\gamma(X)-\gamma(Y)}\leq\norm{X-Y}$.
\item
\label{p:8ii}
A function $\varphi\colon\XX\to\RXX$ is $\SD$-invariant if and
only if $(\forall\ad\in\AD)$
$\varphi\circ\gamma\circ\Lambda_{\ad}=\varphi$.
\end{enumerate}
\end{proposition}

Given a Euclidean space $\HH$ and a set $D\subset\HH$,
the closure of $D$ is denoted by $\overline{D}$,
the distance function to $D$ is
\begin{equation}
d_D\colon\HH\to\RPX\colon
x\mapsto\inf_{y\in D}\norm{x-y},
\end{equation}
and the projector onto $D$ is
\begin{equation}
\Proj_D\colon\HH\to 2^{\HH}\colon
x\mapsto\menge{y\in D}{\norm{x-y}=d_D(x)},
\end{equation}
where $2^{\HH}$ denotes the power set of $\HH$.

\begin{proposition}
\label{p:1}
Suppose that \cref{a:1} is in force and let $D$ be a nonempty
$\SD$-invariant subset of $\XX$. Then the following hold:
\begin{enumerate}
\item
\label{p:1i}
$\overline{\gamma^{-1}(D)}=\gamma^{-1}(\overline{D})$.
\item
\label{p:1ii}
$d_D$ is $\SD$-invariant
and $(\forall\ad\in\AD)$
$d_D\circ\gamma\circ\Lambda_{\ad}=d_D$.
\item
\label{p:1iii}
$d_{\gamma^{-1}(D)}=d_D\circ\gamma$.
\item
\label{p:1iv}
For every $X\in\FH$ and every $Z\in\FH$,
$Z\in\Proj_{\gamma^{-1}(D)}X$
if and only if $\gamma(Z)\in\Proj_D\gamma(X)$
and there exists $\ad\in\AD$ such that
$X=\Lambda_{\ad}\gamma(X)$ and $Z=\Lambda_{\ad}\gamma(Z)$.
\item
\label{p:1v}
For every $X\in\FH$, $\Proj_{\gamma^{-1}(D)}X
=\menge{\Lambda_{\ad}z}{z\in\Proj_D\gamma(X)
\,\,\text{and}\,\,\ad\in\AD_X}$.
\end{enumerate}
\end{proposition}
\begin{proof}
\cref{p:1i}:
It results from the continuity of $\gamma$
(\cref{p:8}\,\cref{p:8i})
that $\gamma^{-1}(\overline{D})$ is a closed subset of $\FH$
that contains $\gamma^{-1}(D)$.
Hence $\overline{\gamma^{-1}(D)}\subset\gamma^{-1}(\overline{D})$.
To establish the reverse inclusion, suppose that
$X\in\gamma^{-1}(\overline{D})$.
Then $\gamma(X)\in\overline{D}$ and there thus exists a sequence
$(x_n)_{n\in\NN}$ in $D$ such that $x_n\to\gamma(X)$.
Next, take $\ad\in\AD_X$ and set
$(\forall n\in\NN)$ $X_n=\Lambda_{\ad}x_n$.
By \cref{c:1}, $(\forall n\in\NN)$ $X_n\in\gamma^{-1}(D)$.
On the other hand, since $\Lambda_{\ad}$ is a linear isometry,
$\norm{X_n-X}
=\norm{\Lambda_{\ad}x_n-\Lambda_{\ad}\gamma(X)}
=\norm{x_n-\gamma(X)}
\to 0$.
Hence $X\in\overline{\gamma^{-1}(D)}$, as desired.

\cref{p:1ii}:
By \cite[Example~13.5]{Livre1},
$d_D^2=\norm{\Cdot}^2-2\brk{\iota_D+\norm{\Cdot}^2/2}^*$,
where $\varphi^*$ denotes the conjugate of a function
$\varphi\colon\XX\to\RXX$. On the other hand, since the function
$\iota_D+\norm{\Cdot}^2/2$ is $\SD$-invariant, it results from
\cite[Corollary~4.4]{PartI} that $(\iota_D+\norm{\Cdot}^2/2)^*$ is
likewise. Thus $d_D$ is $\SD$-invariant.
The latter claim follows from \cref{p:8}\,\cref{p:8ii}.

\cref{p:1iii,p:1iv,p:1v}:
It results from \cite[Proposition~3.5\,(iii)]{PartI} that
$\norm{\Cdot}_{\XX}^2\circ\gamma=\norm{\Cdot}_{\FH}^2$.
Therefore, these claims follow respectively from items (i), (iii),
and (ii) of \cite[Corollary~5.6]{PartI} applied to the
$\SD$-invariant Legendre function
$\psi=\tfrac{1}{2}\norm{\Cdot}_{\XX}^2$
(note that $\SD$ acts on $\XX$ by linear isometries).
\end{proof}

\begin{proposition}
\label{p:5}
Suppose that \cref{a:1} is in force.
Let $(X_n)_{n\in\NN}$ be a sequence in $\FH$ that converges to some
$X\in\FH$ and, for every $n\in\NN$, let $\ad_n\in\AD_{X_n}$.
Suppose that there exists $\ad\in\AD$ such that
$\Lambda_{\ad_n}\to\Lambda_{\ad}$ in $\mathscr{L}(\XX,\FH)$.
Then $\ad\in\AD_X$.
\end{proposition}
\begin{proof}
  Since $\gamma$ is nonexpansive (\cref{p:8}\,\cref{p:8i}), we get
  $\gamma(X_n)\to\gamma(X)$. Thus
  $X
  =\lim X_n
  =\lim\Lambda_{\ad_n}\gamma(X_n)
  =\Lambda_{\ad}\gamma(X)$,
  which confirms that $\ad\in\AD_X$.
\end{proof}

Finally, we point out an important consequence of \cref{a:1} that
will be employed repeatedly in subsequent proofs.

\begin{proposition}
\label{p:54}
Suppose that \cref{a:1} is in force.
Then the set $\set{\Lambda_{\ad}}_{\ad\in\AD}$ is compact
in $\mathscr{L}(\XX,\FH)$.
\end{proposition}
\begin{proof}
Since $\FH$ and $\XX$ are both finite-dimensional,
the assertion follows from the closedness of
$\set{\Lambda_{\ad}}_{\ad\in\AD}$
and the fact that
$(\forall\ad\in\AD)$ $\norm{\Lambda_{\ad}}_{\mathscr{L}(\XX,\FH)}=1$.
\end{proof}

\section{Variational geometry of spectral sets}
\label{sec:3}

Consider \cref{a:1} and let $\varphi\colon\XX\to\RXX$ be
$\SD$-invariant. As mentioned in the introduction, a central goal
of this paper is to describe the regular and limiting
subdifferentials of the spectral function $\varphi\circ\gamma$ in
terms of those of $\varphi$. It will be established in \cref{t:1}
that
\begin{equation}
\label{e:y21z}
(\forall X\in\FH)\quad
\partial^{\#}(\varphi\circ\gamma)(X)
=\menge{\Lambda_{\ad}y}{y\in\partial^{\#}\varphi\brk!{\gamma(X)}\,\,
\text{and}\,\,\ad\in\AD_X},
\end{equation}
where $\partial^{\#}$ stands for either the regular or the limiting
subdifferential operator.
We now trace the path leading to this result.
A common strategy in variational analysis for establishing a result
for functions is to first treat the corresponding statement for sets,
and then apply it to epigraphs to obtain the general case.
Adopting this strategy, we establish in this section the identity
\begin{equation}
\label{e:8kr0}
(\forall X\in\FH)\quad
N_{\gamma^{-1}(D)}^{\#}X
=\menge{\Lambda_{\ad}y}{y\in N_D^{\#}\gamma(X)
\,\,\text{and}\,\,\ad\in\AD_X},
\end{equation}
where $\emp\neq D\subset\XX$ is $\SD$-invariant,
and $N_D^{\#}$ stands for either the regular or the limiting normal
cone operator of $D$. Then, in \cref{sec:41}, we will apply this result
to epigraphs -- within the context of the
``product'' spectral decomposition system from \cref{p:10} --
to obtain the desired formula \cref{e:y21z}.

We proceed by recalling relevant notions from variational analysis
following the standard references \cite{Mord06,Rock09},
to which we refer for background and further details.
Let $\HH$ be a Euclidean space and let $D$ be a nonempty subset of
$\HH$. The regular normal cone operator of $D$ is
\begin{equation}
\widehat{N}_D\colon\HH\to 2^{\HH}\colon
x\mapsto
\begin{cases}
\displaystyle
\Menge4{y\in\HH}{\limsup_{\substack{z\to x\\
z\in D\smallsetminus\set{x}}}
\frac{\scal{z-x}{y}}{\norm{z-x}}\leq 0}
&\text{if}\,\,x\in D,\\
\emp&\text{otherwise}.
\end{cases}
\end{equation}
Evidently,
\begin{equation}
\label{e:56f8}
(\forall x\in D)\quad
\widehat{N}_Dx=\widehat{N}_{\overline{D}}x.
\end{equation}
The limiting normal cone to $D$ at a point $x\in D$,
in symbols $N_Dx$,
is the set of all $y\in\HH$ for which there exists a sequence
$(x_n,y_n)_{n\in\NN}$ in $\gra\widehat{N}_D$ such that
$x_n\to x$ and $y_n\to y$;
here, $\gra M=\menge{(x,y)\in\HH\times\HH}{y\in Mx}$ denotes the
graph of a set-valued operator $M\colon\HH\to 2^{\HH}$.
In addition, we set $(\forall x\in\HH\smallsetminus D)$
$N_Dx=\emp$.

We next record two standard facts on regular normal cones that
will be used repeatedly in our analysis.

\begin{lemma}
\label{l:1}
Let $\HH$ be a Euclidean space,
let $D$ be a nonempty subset of $\HH$,
and let $x\in D$. Then the following hold:
\begin{enumerate}
\item
\label{l:1i}
Let $y\in\widehat{N}_Dx$. Then there exist sequences
$(x_n)_{n\in\NN}$ and $(y_n)_{n\in\NN}$ in $\HH$ such that
\begin{equation}
x_n\to x,
\quad y_n\to y,
\quad\text{and}\quad
(\forall n\in\NN)\,\,
y_n\in\cone\brk!{x_n-\Proj_{\overline{D}}x_n},
\end{equation}
where $(\forall C\in 2^{\HH})$
$\cone C=\bigcup_{\alpha\in\RPP}\alpha C$.
\item\label{l:1ii}
Let $y\in\HH$. Then $y\in\widehat{N}_Dx$ if and only if
$\lim_{\alpha\to 0^+}d_D(x+\alpha y)/\alpha=\norm{y}$.
\end{enumerate}
\end{lemma}
\begin{proof}
\cref{l:1i}:
Combine \cref{e:56f8} and \cite[Theorem~1.6]{Mord06}.

\cref{l:1ii}:
We adapt and simplify the proof of \cite[Lemma~2.1\,(i)]{Drus18} for
this special case. The case $y=0$ is immediate, so we assume
that $y\neq 0$. First, suppose that $y\in\widehat{N}_Dx$ and let
$\varepsilon\in\intv[o]{0}{\norm{y}}$.
Then there exists $\delta\in\zeroun$ such that
$\brk{\forall z\in D\cap B(x;\delta)}$
$\scal{z-x}{y}\leq\varepsilon\norm{z-x}$.
For every $\alpha\in\intv[o]{0}{\delta/(2\norm{y})}$,
we have
\begin{align}
\brk!{\forall z\in D\cap B(x;\delta)}\quad
\norm{x+\alpha y-z}^2
&=\norm{x-z}^2-2\alpha\scal{z-x}{y}
+\alpha^2\norm{y}^2
\nonumber\\
&\geq\norm{x-z}^2-2\alpha\varepsilon\norm{x-z}+\alpha^2\norm{y}^2
\nonumber\\
&=\brk!{\norm{x-z}-\alpha\varepsilon}^2
+\alpha^2(\norm{y}^2-\varepsilon^2)
\nonumber\\
&\geq\alpha^2(\norm{y}^2-\varepsilon^2)
\end{align}
and
\begin{equation}
\brk!{\forall z\in D\smallsetminus B(x;\delta)}\quad
\norm{x+\alpha y-z}
\geq\norm{x-z}-\alpha\norm{y}
>\delta-\alpha\norm{y}
>\alpha\sqrt{\norm{y}^2-\varepsilon^2}.
\end{equation}
Thus
\begin{equation}
\brk!{\forall\alpha\in\intv[o]{0}{\delta/(2\norm{y})}}\quad
d_D(x+\alpha y)\geq\alpha\sqrt{\norm{y}^2-\varepsilon^2},
\end{equation}
which leads to
$\liminf_{\alpha\to 0^+}
d_D(x+\alpha y)/\alpha
\geq\sqrt{\norm{y}^2-\varepsilon^2}$.
Hence, since $\varepsilon\in\intv[o]{0}{\norm{y}}$
was arbitrarily chosen,
it follows that
\begin{equation}
\liminf_{\alpha\to 0^+}
\frac{d_D(x+\alpha y)}{\alpha}
\geq\norm{y}.
\end{equation}
On the other hand, because $x\in D$, we get
\begin{equation}
\limsup_{\alpha\to 0^+}\frac{d_D(x+\alpha y)}{\alpha}
\leq\limsup_{\alpha\to 0^+}\frac{\norm{x+\alpha y-x}}{\alpha}
=\norm{y}.
\end{equation}
Therefore
$\lim_{\alpha\to 0^+}d_D(x+\alpha y)/\alpha=\norm{y}$.
Conversely, suppose that
$\lim_{\alpha\to 0^+}d_D(x+\alpha y)/\alpha=\norm{y}$ but that
$y\notin\widehat{N}_Dx$. By definition, there exist
$\kappa\in\zeroun$ and a sequence
$(z_n)_{n\in\NN}$ in $D\smallsetminus\set{x}$ such that
$z_n\to x$ and $(\forall n\in\NN)$
$\scal{z_n-x}{y}\geq\kappa\norm{z_n-x}$.
Set $(\forall n\in\NN)$ $\alpha_n=\norm{z_n-x}/\kappa$.
Then $\alpha_n\to 0^+$ and, since $\set{z_n}_{n\in\NN}\subset D$, we
deduce that
\begin{align}
(\forall n\in\NN)\quad
\frac{d_D^2(x+\alpha_n y)}{\alpha_n^2}
&\leq\frac{\norm{x+\alpha_n y-z_n}^2}{\alpha_n^2}
\nonumber\\
&=\frac{1}{\alpha_n^2}\brk!{
\norm{x-z_n}^2-2\alpha_n\scal{z_n-x}{y}+\alpha_n^2\norm{y}^2}
\nonumber\\
&\leq\kappa^2-\frac{2\kappa\norm{z_n-x}}{\alpha_n}
+\norm{y}^2
\nonumber\\
&={-}\kappa^2+\norm{y}^2.
\end{align}
Therefore, letting $n\to\pinf$ yields a contradiction.
\end{proof}

In the case of regular normal cones, the inclusion ``$\supset$'' in
\cref{e:8kr0} follows at once from the following lifting property.

\begin{proposition}
\label{p:6}
Suppose that \cref{a:1} is in force.
Let $D$ be a nonempty $\SD$-invariant subset of $\XX$,
let $x$ and $y$ be in $\XX$, and let $\ad\in\AD$.
Then $y\in\widehat{N}_Dx$ if and only if
$\Lambda_{\ad}y\in\widehat{N}_{\gamma^{-1}(D)}(\Lambda_{\ad}x)$.
\end{proposition}
\begin{proof}
Since \cref{c:1} asserts that
$D=\Lambda_{\ad}^{-1}(\gamma^{-1}(D))$, we get
$x\in D$ \,$\Leftrightarrow$\, $\Lambda_{\ad}x\in\gamma^{-1}(D)$.
On the other hand, we derive from items \cref{p:1iii} and
\cref{p:1ii} in \cref{p:1} that
$(\forall\alpha\in\RP)$
$d_{\gamma^{-1}(D)}(\Lambda_{\ad}x+\alpha\Lambda_{\ad}y)
=(d_D\circ\gamma\circ\Lambda_{\ad})(x+\alpha y)
=d_D(x+\alpha y)$.
Therefore, since $\Lambda_{\ad}$ is a linear isometry,
we deduce from \cref{l:1}\,\cref{l:1ii} that
\begin{align}
y\in\widehat{N}_Dx
&\Leftrightarrow
x\in D
\,\,\text{and}\,\,
\lim_{\alpha\to 0^+}\frac{d_D(x+\alpha y)}{\alpha}
=\norm{y}
\nonumber\\
&\Leftrightarrow
\Lambda_{\ad}x\in\gamma^{-1}(D)
\,\,\text{and}\,\,
\lim_{\alpha\to 0^+}\frac{d_{\gamma^{-1}(D)}(\Lambda_{\ad}x
+\alpha\Lambda_{\ad}y)}{\alpha}
=\norm{\Lambda_{\ad}y}
\nonumber\\
&\Leftrightarrow
\Lambda_{\ad}y\in\widehat{N}_{\gamma^{-1}(D)}(\Lambda_{\ad}x),
\end{align}
as claimed.
\end{proof}

We now establish a key identity that underpins the proof of the
inclusion ``$\subset$'' in \cref{e:8kr0}. Since this identity may
also be of independent interest, we state it as a standalone
result.

\begin{proposition}
\label{p:9}
Suppose that \cref{a:1} is in force.
Let $D$ be a nonempty $\SD$-invariant subset of $\XX$
and let $X\in\FH$. Then
\begin{equation}
\cone\brk!{X-\Proj_{\gamma^{-1}(D)}X}
=\menge{\Lambda_{\ad}y}{y\in\cone\brk!{\gamma(X)-\Proj_D\gamma(X)}
\,\,\text{and}\,\,\ad\in\AD_X}.
\end{equation}
\end{proposition}
\begin{proof}
Thanks to the linearity of the operators
$(\Lambda_{\ad})_{\ad\in\AD}$, it suffices to show that
\begin{equation}
\label{e:g1e7}
X-\Proj_{\gamma^{-1}(D)}X
=\menge{\Lambda_{\ad}y}{y\in\gamma(X)-\Proj_D\gamma(X)
\,\,\text{and}\,\,\ad\in\AD_X}.
\end{equation}
Let $Y\in\FH$ and suppose first that $Y\in
X-\Proj_{\gamma^{-1}(D)}X$, that is,
$X-Y\in\Proj_{\gamma^{-1}(D)}X$.
\cref{p:1}\,\cref{p:1iv} asserts that
\begin{equation}
\label{e:kxnw}
\gamma(X-Y)\in\Proj_D\gamma(X)
\end{equation}
and that there exists $\ad\in\AD$ for which
\begin{equation}
\label{e:psjm}
X=\Lambda_{\ad}\gamma(X)
\quad\text{and}\quad
X-Y=\Lambda_{\ad}\gamma(X-Y).
\end{equation}
Since $\Lambda_{\ad}^*\circ\Lambda_{\ad}=\Id_{\XX}$, it follows that
\begin{equation}
\label{e:wkbm}
\gamma(X-Y)
=\Lambda_{\ad}^*(X-Y)
=\Lambda_{\ad}^*X-\Lambda_{\ad}^*Y
=\Lambda_{\ad}^*\brk!{\Lambda_{\ad}\gamma(X)}-\Lambda_{\ad}^*Y
=\gamma(X)-\Lambda_{\ad}^*Y.
\end{equation}
In turn, invoking \cref{e:psjm} once more, we obtain
$X-Y
=\Lambda_{\ad}(\gamma(X)-\Lambda_{\ad}^*Y)
=X-\Lambda_{\ad}(\Lambda_{\ad}^*Y)$, which leads to
$Y=\Lambda_{\ad}\brk{\Lambda_{\ad}^*Y}$.
On the other hand, we deduce from \cref{e:kxnw,e:wkbm} that
$\Lambda_{\ad}^*Y\in\gamma(X)-\Proj_D\gamma(X)$.
We have thus established the inclusion ``$\subset$'' in
\cref{e:g1e7}. Conversely, suppose that $Y=\Lambda_{\ad}y$ for
some $y\in\gamma(X)-\Proj_{D}\gamma(X)$ and some $\ad\in\AD_X$,
and let $z\in\Proj_D\gamma(X)$ be such that $y=\gamma(X)-z$.
\cref{p:1}\,\cref{p:1v} implies that
$\Lambda_{\ad}z\in\Proj_{\gamma^{-1}(D)}X$ and, therefore,
$Y
=\Lambda_{\ad}\brk{\gamma(X)-z}
=X-\Lambda_{\ad}z
\in X-\Proj_{\gamma^{-1}(D)}X$,
which completes the proof.
\end{proof}

We are now in a position to establish the formula \cref{e:8kr0}.

\begin{proposition}
\label{p:2}
Suppose that \cref{a:1} is in force.
Let $D$ be a nonempty $\SD$-invariant subset of $\XX$
and let $X\in\FH$. Then the following hold:
\begin{enumerate}
\item
\label{p:2i}
$\widehat{N}_{\gamma^{-1}(D)}X
=\menge{\Lambda_{\ad}y}{y\in\widehat{N}_D\gamma(X)
\,\,\text{and}\,\,\ad\in\AD_X}$.
\item
\label{p:2ii}
$N_{\gamma^{-1}(D)}X
=\menge{\Lambda_{\ad}y}{y\in N_D\gamma(X)
\,\,\text{and}\,\,\ad\in\AD_X}$.
\end{enumerate}
\end{proposition}
\begin{proof}
We assume henceforth that
\begin{equation}
\label{e:wp32}
X\in\gamma^{-1}(D)
\end{equation}
since otherwise all the normal cones in question are empty and the
assertions are therefore trivial. Since $d_D$ is $\SD$-invariant
(\cref{p:1}\,\cref{p:1ii}), we obtain
$(\forall x\in\XX)(\forall\sd\in\SD)$
$x\in\overline{D}$
\,$\Leftrightarrow$\,
$d_D(x)=0$
\,$\Leftrightarrow$\,
$d_D(\sd\cdot x)=0$
\,$\Leftrightarrow$\,
$\sd\cdot x\in\overline{D}$.
Therefore, $\overline{D}$ is $\SD$-invariant.
Now let $Y\in\FH.$

\cref{p:2i}:
The inclusion ``$\supset$'' follows from \cref{p:6} applied to
$x=\gamma(X)$. To establish the converse, suppose that
$Y\in\widehat{N}_{\gamma^{-1}(D)}X$.
\cref{l:1}\,\cref{l:1i} states that there exist sequences
$(X_n)_{n\in\NN}$ and $(Y_n)_{n\in\NN}$ in $\FH$ such that
\begin{equation}
\label{e:erw9}
X_n\to X,
\quad
Y_n\to Y,
\quad\text{and}\quad
(\forall n\in\NN)\,\,
Y_n\in\cone\brk!{X_n-\Proj_{\overline{\gamma^{-1}(D)}}X_n}.
\end{equation}
In turn, we learn from \cref{p:1}\,\cref{p:1i} and \cref{p:9}
(applied to the $\SD$-invariant set $\overline{D}$) that
\begin{align}
(\forall n\in\NN)\quad
Y_n
&\in\cone\brk!{X_n-\Proj_{\gamma^{-1}(\overline{D})}X_n}
\nonumber\\
&=\menge{\Lambda_{\ad}y}{y\in\cone\brk!{\gamma(X_n)
-\Proj_{\overline{D}}\gamma(X_n)}\,\,\text{and}\,\,\ad\in\AD_{X_n}}.
\end{align}
Hence, we obtain sequences $(y_n)_{n\in\NN}$ in $\XX$
and $(\ad_n)_{n\in\NN}$ in $\AD$ such that
\begin{equation}
\label{e:hmt5}
(\forall n\in\NN)\quad
y_n\in\cone\brk!{\gamma(X_n)-\Proj_{\overline{D}}\gamma(X_n)},
\quad
\ad_n\in\AD_{X_n},
\quad\text{and}\quad
Y_n=\Lambda_{\ad_n}y_n.
\end{equation}
Since the operators $(\Lambda_{\ad})_{\ad\in\AD}$ are linear
isometries, it follows from \cref{e:erw9} that
$\sup_{n\in\NN}\norm{y_n}=\sup_{n\in\NN}\norm{Y_n}<\pinf$.
Thus, we deduce from \cref{p:54} that there exist a strictly
increasing sequence $(k_n)_{n\in\NN}$ in $\NN$, together with
$y\in\XX$ and $\ad\in\AD$ such that
\begin{equation}
\label{e:5gtr}
y_{k_n}\to y
\quad\text{and}\quad
\Lambda_{\ad_{k_n}}\to\Lambda_{\ad}.
\end{equation}
In turn, we derive from \cref{e:erw9,e:hmt5} that
$Y
=\lim Y_{k_n}
=\lim\Lambda_{\ad_{k_n}}y_{k_n}
=\Lambda_{\ad}y$.
At the same time, invoking \cref{e:erw9,e:hmt5,e:5gtr},
we deduce from \cref{p:5} that $\ad\in\AD_X$.
Therefore
$\Lambda_{\ad}y
=Y
\in\widehat{N}_{\gamma^{-1}(D)}X
=\widehat{N}_{\gamma^{-1}(D)}(\Lambda_{\ad}\gamma(X))$.
Consequently, \cref{p:6} yields $y\in\widehat{N}_D\gamma(X)$.

\cref{p:2ii}:
Assume that $Y\in N_{\gamma^{-1}(D)}X$ and let
$(X_n,Y_n)_{n\in\NN}$ be a sequence in
$\gra\widehat{N}_{\gamma^{-1}(D)}$
such that $X_n\to X$ and $Y_n\to Y$.
It results from \cref{p:2i} that there exist sequences
$(y_n)_{n\in\NN}$ in $\XX$ and $(\ad_n)_{n\in\NN}$ in $\AD$ such
that
\begin{equation}
\label{e:mjl1}
(\forall n\in\NN)\quad
Y_n=\Lambda_{\ad_n}y_n,
\quad
y_n\in\widehat{N}_D\gamma(X_n),
\quad\text{and}\quad
\ad_n\in\AD_{X_n}.
\end{equation}
Note that $\sup_{n\in\NN}\norm{y_n}=\sup_{n\in\NN}\norm{Y_n}<\pinf$.
Combining with \cref{p:54},
we obtain a strictly increasing sequence
$(k_n)_{n\in\NN}$ in $\NN$, a point $y\in\XX$, and an element
$\ad\in\AD$ such that
\begin{equation}
\label{e:mjl2}
y_{k_n}\to y
\quad\text{and}\quad
\Lambda_{\ad_{k_n}}\to\Lambda_{\ad}.
\end{equation}
Moreover, the nonexpansiveness of $\gamma$ (\cref{p:8}\,\cref{p:8i})
gives $\gamma(X_{k_n})\to\gamma(X)$.
Hence, \cref{e:mjl1,e:mjl2} entail that $y\in N_D\gamma(X)$.
At the same time, we have
$Y
=\lim Y_{k_n}
=\lim\Lambda_{\ad_{k_n}}y_{k_n}
=\Lambda_{\ad}y$, while \cref{p:5} entails that $\ad\in\AD_X$.
Finally, to establish the reverse inclusion,
assume that $Y=\Lambda_{\bd}v$ for some $v\in N_D\gamma(X)$ and
some $\bd\in\AD_X$. Let $(z_n,v_n)_{n\in\NN}$ be a sequence in
$\HH\times\HH$ such that
$z_n\to\gamma(X)$,
$v_n\to v$,
and $(\forall n\in\NN)$ $v_n\in\widehat{N}_Dz_n$.
By \cref{p:6}, $(\forall n\in\NN)$
$\Lambda_{\bd}v_n\in\widehat{N}_{\gamma^{-1}(D)}(\Lambda_{\bd}z_n)$.
Hence, since the continuity of $\Lambda_{\bd}$ gives
$\Lambda_{\bd}z_n\to\Lambda_{\bd}\gamma(X)=X$
and $\Lambda_{\bd}v_n\to\Lambda_{\bd}v=Y$,
we conclude that $Y\in N_{\gamma^{-1}(D)}X$.
\end{proof}

\section{Generalized subgradients of spectral functions}
\label{sec:4}

The goal of this section is to express several first-order
variational objects of a spectral function $\varphi\circ\gamma$ in
terms of those associated with the invariant function
$\varphi\colon\XX\to\RXX$ (as long as \cref{a:1} holds). Toward
this end, we first study the case of regular and limiting
subdifferentials by applying the results of \cref{sec:3} to
epigraphs within a suitable spectral decomposition system. We then
employ these results to characterize Fr\'{e}chet differentiability
and Clarke subdifferentials of spectral functions.

\subsection{Regular and limiting subdifferentials of spectral
functions}
\label{sec:41}

We start by recalling the notions of regular and limiting
subdifferentials.
Let $\HH$ be a Euclidean space and let $f\colon\HH\to\RXX$.
The regular subdifferential of $f$ is
\begin{equation}
\widehat{\partial}f\colon\HH\to 2^{\HH}\colon
x\mapsto
\begin{cases}
\displaystyle
\Menge4{y\in\HH}{\liminf_{\substack{z\to x\\ z\neq x}}
\frac{f(z)-f(x)-\scal{z-x}{y}}{\norm{z-x}}\geq 0}
&\text{if}\,\,f(x)\in\RR,\\
\emp&\text{otherwise}.
\end{cases}
\end{equation}
The limiting subdifferential of $f$ at a point
$x\in\HH$ such that $f(x)\in\RR$, in symbols $\partial f(x)$,
is the set of all $y\in\HH$ for which there exists a sequence
$(x_n,y_n)_{n\in\NN}$ in $\gra\widehat{\partial}f$ such that
$x_n\to x$, $f(x_n)\to f(x)$, and $y_n\to y$.
Further, for every $x\in\HH$ such that $f(x)\in\set{\pm\infty}$,
we set $\partial f(x)=\emp$.
For every $x\in\HH$ such that $f(x)\in\RR$, we have
\begin{equation}
\label{e:Nepi}
\partial^{\#}f(x)
=\menge{y\in\HH}{(y,-1)\in N_{\epi f}^{\#}\brk!{x,f(x)}},
\end{equation}
where
$\partial^{\#}f$ denotes either the regular or the limiting
subdifferential operator, $N^{\#}$ denotes the corresponding normal
cone operator, and
$\epi f=\menge{(y,\eta)\in\HH\times\RR}{f(y)\leq\eta}$
is the epigraph of $f$; see \cite[Theorem~8.9]{Rock09}.

As a byproduct of our analysis, we also establish the property
\begin{equation}
\label{e:sqe2}
(\forall x\in\XX)(\forall y\in\XX)(\forall\ad\in\AD)\quad
y\in\partial^{\#}\varphi(x)
\quad\Rightarrow\quad
\Lambda_{\ad}y\in\partial^{\#}(\varphi\circ\gamma)(\Lambda_{\ad}x).
\end{equation}
This can be viewed as a strengthening of the inclusion
``$\supset$'' in \cref{e:y21z}: It allows one to lift subgradients
of an invariant function at an arbitrary point $x\in\XX$,
not just at spectra of points $X\in\FH$.
This lifting property will be used in \cref{c:5} to relate
Fr\'{e}chet differentiability of $\varphi$ to that of
$\varphi\circ\gamma$.

The following product spectral decomposition system is a key device
for applying the normal-cone formula of \cref{sec:3} to epigraphs.

\begin{proposition}
\label{p:10}
Suppose that \cref{a:1} is in force.
Let $\SD$ act on the direct sum $\XX\oplus\RR$ via
$\brk{\sd,(x,\xi)}\mapsto(\sd\cdot x,\xi)$, and define
\begin{equation}
\begin{cases}
\boldsymbol{\gamma}\colon\FH\oplus\RR\to\XX\oplus\RR\colon
(X,\xi)\mapsto\brk!{\gamma(X),\xi}\\
(\forall\ad\in\AD)\,\,
\boldsymbol{\Lambda}_{\ad}\colon\XX\oplus\RR\to\FH\oplus\RR\colon
(x,\xi)\mapsto(\Lambda_{\ad}x,\xi).
\end{cases}
\end{equation}
Then the following hold:
\begin{enumerate}
\item
\label{p:10i}
$(\XX\oplus\RR,\SD,\boldsymbol{\gamma},
(\boldsymbol{\Lambda}_{\ad})_{\ad\in\AD})$
is a spectral decomposition system for $\FH\oplus\RR$.
\item
\label{p:10ii}
The set $\set{\boldsymbol{\Lambda}_{\ad}}_{\ad\in\AD}$
is closed in $\mathscr{L}(\XX\oplus\RR,\FH\oplus\RR)$.
\item
\label{p:10iii}
Let $\varphi\colon\XX\to\RXX$ be $\SD$-invariant.
Then $\epi\varphi$ is an $\SD$-invariant subset of
$\XX\oplus\RR$ and
$\epi(\varphi\circ\gamma)=\boldsymbol{\gamma}^{-1}(\epi\varphi)$.
\end{enumerate}
\end{proposition}
\begin{proof}
\cref{p:10i}:
Let $\tau$ be the spectral-induced ordering mapping
of the system $(\XX,\SD,\gamma,(\Lambda_{\ad})_{\ad\in\AD})$
(see property~\cref{d:1a} in \cref{d:1}).
Then $(\XX\oplus\RR,\SD,\boldsymbol{\gamma},
(\boldsymbol{\Lambda}_{\ad})_{\ad\in\AD})$ satisfies
property~\cref{d:1a} in \cref{d:1} with
$\boldsymbol{\tau}\colon\XX\oplus\RR\to\XX\oplus\RR\colon
(x,\xi)\mapsto(\tau(x),\xi)$.
Next, properties~\cref{d:1b} and \cref{d:1c} in \cref{d:1} follow
from the corresponding properties of the original system. Indeed,
\begin{equation}
\brk!{\forall(X,\xi)\in\FH\oplus\RR}(\forall\ad\in\AD_X)\quad
(X,\xi)
=\brk!{\Lambda_{\ad}\gamma(X),\xi}
=\boldsymbol{\Lambda}_{\ad}\brk!{\gamma(X),\xi}
=\boldsymbol{\Lambda}_{\ad}\boldsymbol{\gamma}(X,\xi)
\end{equation}
and
\begin{align}
\brk!{\forall(X,\xi)\in\FH\oplus\RR}
\brk!{\forall(Y,\eta)\in\FH\oplus\RR}\quad
\scal!{(X,\xi)}{(Y,\eta)}
&=\scal{X}{Y}+\xi\eta
\nonumber\\
&\leq\scal{\gamma(X)}{\gamma(Y)}+\xi\eta
\nonumber\\
&=\scal!{\brk!{\gamma(X),\xi}}{\brk!{\gamma(Y),\eta}}
\nonumber\\
&=\scal!{\boldsymbol{\gamma}(X,\xi)}{\boldsymbol{\gamma}(Y,\eta)}.
\end{align}

\cref{p:10ii}:
Let $(\ad_n)_{n\in\NN}$ be a sequence in $\AD$ such that
$(\boldsymbol{\Lambda}_{\ad_n})_{n\in\NN}$ converges in
$\mathscr{L}(\XX\oplus\RR,\FH\oplus\RR)$ to some operator
$\boldsymbol{\Lambda}$.
\cref{p:54} ensures that there exist a strictly increasing sequence
$(k_n)_{n\in\NN}$ in $\NN$ and an index $\ad\in\AD$ such that
$\Lambda_{\ad_{k_n}}\to\Lambda_{\ad}$ in
$\mathscr{L}(\XX,\FH)$.
Therefore, taking limits in the product space $\FH\oplus\RR$, we
obtain $\brk{\forall(x,\xi)\in\XX\oplus\RR}$
$\boldsymbol{\Lambda}(x,\xi)
=\lim\boldsymbol{\Lambda}_{\ad_{k_n}}(x,\xi)
=\lim\brk{\Lambda_{\ad_{k_n}}x,\xi}
=\brk{\Lambda_{\ad} x,\xi}
=\boldsymbol{\Lambda}_{\ad}(x,\xi)$.

\cref{p:10iii}:
Since $\varphi$ is $\SD$-invariant,
\begin{align}
\brk!{\forall(x,\xi)\in\XX\oplus\RR}(\forall\sd\in\SD)\quad
(x,\xi)\in\epi\varphi
&\Leftrightarrow
\varphi(\sd\cdot x)=\varphi(x)\leq\xi
\nonumber\\
&\Leftrightarrow
\sd\cdot(x,\xi)=(\sd\cdot x,\xi)\in\epi\varphi,
\end{align}
which shows that $\epi\varphi$ is an $\SD$-invariant subset of
$\XX\oplus\RR$. At the same time, we have
\begin{align}
\label{e:41wq}
\boldsymbol{\gamma}^{-1}(\epi\varphi)
&=\menge{(X,\xi)\in\FH\oplus\RR}{\boldsymbol{\gamma}(X,\xi)
\in\epi\varphi}
\nonumber\\
&=\menge{(X,\xi)\in\FH\oplus\RR}{
\brk!{\gamma(X),\xi}\in\epi\varphi}
\nonumber\\
&=\menge{(X,\xi)\in\FH\oplus\RR}{\varphi\brk!{\gamma(X)}\leq\xi}
\nonumber\\
&=\epi(\varphi\circ\gamma),
\end{align}
which completes the proof.
\end{proof}

Having assembled the necessary tools, we now state and prove
the formulae for regular and limiting subdifferentials of spectral
functions.

\begin{theorem}
\label{t:1}
Suppose that \cref{a:1} is in force.
Let $\varphi\colon\XX\to\RXX$ be $\SD$-invariant and
let $X\in\FH$.
Then the following hold:
\begin{enumerate}
\item\label{t:1i}
$\widehat{\partial}(\varphi\circ\gamma)(X)
=\menge{\Lambda_{\ad}y}{y\in\widehat{\partial}\varphi(\gamma(X))\,\,
\text{and}\,\,\ad\in\AD_X}$.
\item\label{t:1ii}
$\partial(\varphi\circ\gamma)(X)
=\menge{\Lambda_{\ad}y}{y\in\partial\varphi(\gamma(X))\,\,
\text{and}\,\,\ad\in\AD_X}$.
\end{enumerate}
\end{theorem}
\begin{proof}
We consider the notation and setting of \cref{p:10},
and work with the spectral decomposition system
$\mathfrak{P}=(\XX\oplus\RR,\SD,\boldsymbol{\gamma},
(\boldsymbol{\Lambda}_{\ad})_{\ad\in\AD})$ for $\FH\oplus\RR$.
It will be convenient to set
$x=\gamma(X)$ and $\xi=\varphi\brk{\gamma(X)}$.
If $\xi\in\set{\pm\infty}$, then all the subdifferentials in
question are empty by definition, and the assertions hold trivially.
Therefore, we assume henceforth that $\xi\in\RR$.
Note that $\boldsymbol{\gamma}(X,\xi)=(x,\xi)$.

\cref{t:1i}:
\cref{p:10}\,\cref{p:10iii} states that $\epi\varphi$ is an
$\SD$-invariant subset of $\XX\oplus\RR$
and that
$\epi(\varphi\circ\gamma)=\boldsymbol{\gamma}^{-1}(\epi\varphi)$.
Hence, it follows from \cref{e:Nepi} and \cref{p:2}\,\cref{p:2i}
(applied to the system $\mathfrak{P}$ and the $\SD$-invariant set
$\epi\varphi$) that
\begin{align}
&(\forall Y\in\FH)\quad
Y\in\widehat{\partial}(\varphi\circ\gamma)(X)
\nonumber\\
&\hskip 21mm
\Leftrightarrow
(Y,-1)\in\widehat{N}_{\epi(\varphi\circ\gamma)}(X,\xi)
=\widehat{N}_{\boldsymbol{\gamma}^{-1}(\epi\varphi)}(X,\xi)
\nonumber\\
&\hskip 21mm
\Leftrightarrow
\brk!{\exi(y,\eta)\in\XX\oplus\RR}(\exi\ad\in\AD)\,\,
\begin{cases}
(y,\eta)
\in\widehat{N}_{\epi\varphi}\brk!{\boldsymbol{\gamma}\brk{X,\xi}}\\
\brk{X,\xi}
=\boldsymbol{\Lambda}_{\ad}\boldsymbol{\gamma}\brk{X,\xi}
=\brk{\Lambda_{\ad}x,\xi}\\
(Y,-1)
=\boldsymbol{\Lambda}_{\ad}(y,\eta)
=\brk{\Lambda_{\ad}y,\eta}
\end{cases}
\nonumber\\
&\hskip 21mm
\Leftrightarrow
\brk{\exi y\in\XX}(\exi\ad\in\AD)\,\,
\begin{cases}
(y,-1)\in\widehat{N}_{\epi\varphi}(x,\xi)\\
X=\Lambda_{\ad}x\,\,\text{and}\,\,
Y=\Lambda_{\ad}y
\end{cases}
\nonumber\\
&\hskip 21mm
\Leftrightarrow
\brk!{\exi y\in\widehat{\partial}\varphi(x)}(\exi\ad\in\AD_X)\,\,
Y=\Lambda_{\ad}y,
\end{align}
as desired.

\cref{t:1ii}:
Argue as in \cref{t:1i} and use \cref{p:2}\,\cref{p:2ii} instead
of \cref{p:2}\,\cref{p:2i}.
\end{proof}

\begin{remark}
\label{r:6}
\cref{t:1} unifies several existing results that describe the
regular and limiting subdifferentials of a spectral function in
terms of the corresponding subdifferentials of the associated
invariant function and the underlying spectral mapping. More
precisely:
\begin{enumerate}
\item
In the Euclidean Jordan algebra framework of \cref{ex:3},
\cref{t:1} reduces to \cite[Theorem~17]{Lourenco20} (see
\cite[Theorems~8.5 and 9.1]{Sendov07} for special cases).
\item
In the matrix frameworks of \cref{ex:4,ex:5},
\cref{t:1} reduces to \cite[Theorem~6]{Lewis99b}
(see also \cite[Theorem~4.2]{Drus18})
and \cite[Theorem~7.1]{Lewis05a}, respectively.
\item
To the best of our knowledge, \cref{t:1} is new in the normal
decomposition framework of \cref{ex:2}. In particular, it provides
an affirmative answer, within this framework, to the
question raised by Lewis in the paragraph following
\cite[Theorem~7.2]{Lewis03} concerning a full unification of
subdifferential formulae for spectral functions.
\end{enumerate}
We emphasize that the proofs of the above-mentioned results rely on
intricate techniques that are specific to the underlying framework,
such as commutativity and simultaneous diagonalization arguments,
and directional derivatives of the corresponding spectral mapping.
These techniques do not appear to extend directly to the abstract
framework governed by \cref{a:1}.
Our proof follows a different route by first
establishing normal-cone formulae for spectral sets and then
applying them to epigraphs through the product spectral decomposition
system constructed in \cref{p:10}.
\end{remark}

\begin{remark}
Suppose that \cref{a:1} is in force and denote by
$\widetilde{\partial}$ the convex subdifferential operator.
It was established in \cite[Proposition~4.6]{PartI} that,
for every proper $\SD$-invariant function $\varphi\colon\XX\to\RX$,
we have
\begin{equation}
(\forall X\in\FH)(\forall Y\in\FH)\quad
Y\in\widetilde{\partial}(\varphi\circ\gamma)(X)
\quad\Leftrightarrow\quad
\begin{cases}
\gamma(Y)\in\widetilde{\partial}\varphi\brk!{\gamma(X)}\\
(\exi\ad\in\AD)\,\,
X=\Lambda_{\ad}\gamma(X)
\,\,\text{and}\,\,
Y=\Lambda_{\ad}\gamma(Y).
\end{cases}
\end{equation}
By contrast, one cannot assert that
$Y\in\widehat{\partial}(\varphi\circ\gamma)(X)$
$\Rightarrow$
$\gamma(Y)\in\widehat{\partial}\varphi(\gamma(X))$.
To see this, consider the setting of \cref{ex:4} with
$\KB=\RR$ and $N=2$, and, following 
\cite[Example~3.2]{Daniilidis08}, we consider the function
$\varphi\colon\RR^2\to\RR\colon(\xi_1,\xi_2)\mapsto\xi_1\xi_2$.
Then $\varphi\circ\lambda$ is Fr\'echet differentiable at
$X=\Diag(1,2)$ with
$\nabla(\varphi\circ\lambda)(X)=\Diag(2,1)$.
Hence
$\widehat{\partial}(\varphi\circ\lambda)(X)=\set{\Diag(2,1)}$.
On the other hand, since
$\nabla\varphi\brk{\lambda(X)}
=\nabla\varphi(2,1)
=(1,2)$, we infer that
$\widehat{\partial}\varphi(\lambda(X))=\set{(1,2)}
\notni(2,1)=\lambda\brk{\Diag(2,1)}$.
\end{remark}

We end this subsection with a strengthening of the inclusions
``$\supset$'' in \cref{t:1}. In essence, given prior knowledge of
subgradients of an invariant function at a point $x\in\XX$, the
result provides an explicit way to construct subgradients of the
induced spectral function at the points
$\set{\Lambda_{\ad}x}_{\ad\in\AD}$.

\begin{proposition}
\label{p:45}
Suppose that \cref{a:1} is in force.
Let $\varphi\colon\XX\to\RXX$ be $\SD$-invariant,
let $x$ and $y$ be in $\XX$, and let $\ad\in\AD$.
Then the following hold:
\begin{enumerate}
\item\label{p:45i}
$y\in\widehat{\partial}\varphi(x)$ if and only if
$\Lambda_{\ad}y\in\widehat{\partial}(\varphi\circ\gamma)
(\Lambda_{\ad}x)$.
\item\label{p:45ii}
Suppose that $y\in\partial\varphi(x)$. Then
$\Lambda_{\ad}y\in\partial(\varphi\circ\gamma)
(\Lambda_{\ad}x)$.
\end{enumerate}
\end{proposition}
\begin{proof}
\cref{p:45i}:
As in the proof of \cref{t:1},
we employ the spectral decomposition system
$\mathfrak{P}=(\XX\oplus\RR,\SD,\boldsymbol{\gamma},
(\boldsymbol{\Lambda}_{\ad})_{\ad\in\AD})$ constructed in
\cref{p:10}.
Then $\epi\varphi$ is an $\SD$-invariant subset of $\XX\oplus\RR$ and
$\boldsymbol{\gamma}^{-1}(\epi\varphi)
=\epi(\varphi\circ\gamma)$.
In turn, we derive from \cref{p:6} (applied to the system
$\mathfrak{P}$ and the $\SD$-invariant set $\epi\varphi$)
and the identity $\varphi=\varphi\circ\gamma\circ\Lambda_{\ad}$
(\cref{p:8}\,\cref{p:8ii}) that
\begin{align}
y\in\widehat{\partial}\varphi(x)
&\Leftrightarrow
\varphi(x)\in\RR
\,\,\text{and}\,\,
(y,-1)\in\widehat{N}_{\epi\varphi}\brk!{x,\varphi(x)}
\nonumber\\
&\Leftrightarrow
\varphi(x)\in\RR
\,\,\text{and}\,\,
\boldsymbol{\Lambda}_{\ad}(y,-1)\in
\widehat{N}_{\boldsymbol{\gamma}^{-1}(\epi\varphi)}
\brk!{\boldsymbol{\Lambda}_{\ad}\brk{x,\varphi(x)}}
\nonumber\\
&\Leftrightarrow
\varphi(x)\in\RR
\,\,\text{and}\,\,
(\Lambda_{\ad}y,-1)\in
\widehat{N}_{\epi(\varphi\circ\gamma)}
\brk!{\Lambda_{\ad}x,\varphi(x)}
\nonumber\\
&\Leftrightarrow
(\varphi\circ\gamma)(\Lambda_{\ad}x)\in\RR
\,\,\text{and}\,\,
(\Lambda_{\ad}y,-1)\in
\widehat{N}_{\epi(\varphi\circ\gamma)}
\brk!{\Lambda_{\ad}x,(\varphi\circ\gamma)(\Lambda_{\ad}x)}
\nonumber\\
&\Leftrightarrow
\Lambda_{\ad}y\in
\widehat{\partial}(\varphi\circ\gamma)(\Lambda_{\ad}x),
\end{align}
as announced.

\cref{p:45ii}:
Use the definition of limiting subdifferentials, \cref{p:45i}, and
the identity $\varphi\circ\gamma\circ\Lambda_{\ad}=\varphi$.
\end{proof}

\subsection{Fr\'echet differentiability of spectral functions}

The results of \cref{sec:41} allow us to fully characterize the Fr\'echet
differentiability of a spectral function through that of the associated
invariant function.
Recall that, given a Euclidean space $\HH$, a function
$f\colon\HH\to\RXX$ is Fr\'echet differentiable at a point
$x\in\HH$ at which $f(x)\in\RR$ if there exists a unique point in
$\HH$, denoted by $\nabla f(x)$, such that
\begin{equation}
  \lim_{\substack{z\to x\\ z\neq x}}\frac{\abs!{
      f(z)-f(x)-\scal{z-x}{\nabla f(x)}}}{\norm{z-x}}=0;
\end{equation}
moreover, a characterization of Fr\'echet differentiability
in terms of the regular subdifferentials is
\begin{equation}
\label{e:fgf0}
\text{$f$ is Fr\'echet differentiable at $x$}
\quad\Leftrightarrow\quad
\widehat{\partial}(\pm f)(x)\neq\emp;
\end{equation}
see, e.g., \cite[Proposition~1.87]{Mord06}.

\begin{corollary}
\label{c:5}
Suppose that \cref{a:1} is in force and let $\varphi\colon\XX\to\RXX$
be $\SD$-invariant.
\begin{enumerate}
\item\label{c:5i}
Let $x\in\XX$ and $\ad\in\AD$.
Suppose that $\varphi(x)\in\RR$.
Then $\varphi\circ\gamma$ is Fr\'echet differentiable at
$\Lambda_{\ad}x$ if and only if $\varphi$ is Fr\'echet
differentiable at $x$, in which case
\begin{equation}
\nabla(\varphi\circ\gamma)(\Lambda_{\ad}x)
=\Lambda_{\ad}\brk!{\nabla\varphi(x)}.
\end{equation}
\item\label{c:5ii}
Let $X\in\FH$ and suppose that $\varphi(\gamma(X))\in\RR$.
Then $\varphi\circ\gamma$ is Fr\'echet differentiable at $X$ if and
only if $\varphi$ is Fr\'echet differentiable at $\gamma(X)$, in
which case
\begin{equation}
\label{e:d2a5}
(\forall\ad\in\AD_X)\quad
\nabla(\varphi\circ\gamma)(X)
=\Lambda_{\ad}\brk!{\nabla\varphi\brk!{\gamma(X)}}.
\end{equation}
\end{enumerate}
\end{corollary}
\begin{proof}
\cref{c:5i}:
If $\varphi\circ\gamma$ is Fr\'echet differentiable at
$\Lambda_{\ad}x$, then the Fr\'echet differentiability of $\varphi$
at $x$ follows from the identity
$\varphi=(\varphi\circ\gamma)\circ\Lambda_{\ad}$
(\cref{p:8}\,\cref{p:8ii}) and the chain rule.
Now assume that $\varphi$ is Fr\'echet differentiable at $x$.
Then $\pm\nabla\varphi(x)\in\widehat{\partial}(\pm\varphi)(x)$
\cite[Exercise~8.8\,(a)]{Rock09}.
Hence, applying \cref{p:45}\,\cref{p:45i} to the functions
$\pm\varphi$ and using the linearity of $\Lambda_{\ad}$, we obtain
$\pm\Lambda_{\ad}\brk{\nabla\varphi(x)}
\in\widehat{\partial}(\pm\varphi\circ\gamma)(\Lambda_{\ad}x)$.
Therefore, we conclude via \cref{e:fgf0} that $\varphi\circ\gamma$ is
Fr\'echet differentiable at $\Lambda_{\ad}x$ and that
$\nabla(\varphi\circ\gamma)(\Lambda_{\ad}x)=
\Lambda_{\ad}\brk{\nabla\varphi(x)}$.

\cref{c:5ii}:
Apply \cref{c:5i} to $x=\gamma(X)$.
\end{proof}

\begin{remark}
\label{r:4}
\cref{c:5} recovers several results on the Fr\'echet
differentiability of spectral functions.
More precisely:
\begin{enumerate}
\item In the Euclidean Jordan algebra framework of \cref{ex:3}, we
recover at once \cite[Theorem~38]{Baes07} and
\cite[Theorem~4.1]{Sun08} from \cref{c:5}\,\cref{c:5ii}.
\item In the case of Hermitian matrices (\cref{ex:4}),
\cref{c:5}\,\cref{c:5i} yields
\cite[Theorem~2.4 and Corollary~2.5]{Lewis96c}, while
\cref{c:5}\,\cref{c:5ii} yields \cite[Theorem~1.1]{Lewis96c}.
\end{enumerate}
\end{remark}

\subsection{Clarke subdifferentials of spectral functions}

For locally Lipschitz continuous functions, the Clarke subdifferential (or
generalized gradient) is especially useful in optimization due to its strong
calculus and explicit finite-dimensional characterization.
In finite-dimensional settings, the Clarke subdifferential
furnishes generalized derivatives that can be used in superlinearly
convergent semismooth Newton methods; see \cite{Mifflin:1977,Qi:1993,Qi:1993a}
and \cite[Chapter 14]{ClasonValkonen}.

In the present setting, Clarke subdifferentials also provide
the tool for a generalization of Lidski\u{\i}'s theorem, which is
the topic of \cref{sec:5}. Thus, the purpose of this subsection is to
characterize the Clarke subdifferential of a spectral function via
that of the associated invariant function.

We first recall the definition. Let $\HH$ be a Euclidean space and let
$f\colon\HH\to\RXX$ be locally Lipschitz continuous near a point $x\in\HH$,
that is, $f(x)\in\RR$ and there exist $\varepsilon\in\zeroun$ and
$\kappa\in\RP$ such that
\begin{equation}
\brk!{\forall y\in B(x;\varepsilon)}
\brk!{\forall z\in B(x;\varepsilon)}\quad
\abs{f(y)-f(z)}\leq\kappa\norm{y-z},
\end{equation}
where $B(x;\varepsilon)=\menge{y\in\HH}{\norm{x-y}\leq\varepsilon}$.
Given $u\in\HH$, the Clarke generalized directional derivative
of $f$ at $x$ in the direction $u$ is
\begin{equation}
f^{\mathord{\circ}}(x;u)
=\limsup_{\substack{z\to x\\\alpha\to 0^+}}
\frac{f(z+\alpha u)-f(z)}{\alpha}.
\end{equation}
In turn, the Clarke subdifferential of $f$ at $x$ is
\begin{equation}
\overline{\partial}f(x)
=\menge{y\in\HH}{(\forall u\in\HH)\,\,
\scal{u}{y}\leq f^{\mathord{\circ}}(x;u)},
\end{equation}
which is always nonempty, compact, and convex
\cite[Proposition~2.1.2\,(a)]{Clarke90}.
The following characterization of Clarke subdifferentials will be
crucial in this section and also in \cref{sec:5}.

\begin{lemma}
\label{l:2}
Let $\HH$ be a Euclidean space,
let $f\colon\HH\to\RXX$ be locally Lipschitz continuous near a point
$x\in\HH$, and denote by $S_f$ the set of points at which $f$ is
Fr\'echet differentiable.
Then the following hold:
\begin{enumerate}
\item\label{l:2i}
There exists a neighborhood $V$ of $x$ such that
$V\smallsetminus S_f$ is a Lebesgue null set, i.e., a Lebesgue
measurable subset of $\HH$ of Lebesgue measure $0$.
\item\label{l:2ii}
Let $\Omega\subset\HH$ be a Lebesgue null set, and let $D_{\Omega}$
be the set of all $y\in\HH$ for which there
exists a sequence $(x_n)_{n\in\NN}$ in $S_f\smallsetminus\Omega$
such that $x_n\to x$ and $\nabla f(x_n)\to y$.
Then $\overline{\partial}f(x)=\conv D_{\Omega}\neq\emp$.
\end{enumerate}
\end{lemma}
\begin{proof}
\cref{l:2i}:
This is a consequence of Rademacher's theorem.

\cref{l:2ii}:
See \cite[Theorem~2.5.1]{Clarke90}.
\end{proof}

We can now state and prove the main result of this subsection.

\begin{proposition}
\label{p:3}
Suppose that \cref{a:1} is in force, and
let $\varphi\colon\XX\to\RXX$ be $\SD$-invariant
and locally Lipschitz continuous near the spectrum
$\gamma(X)$ of a point $X\in\FH$. Then
\begin{equation}
\label{e:udqz}
\overline{\partial}(\varphi\circ\gamma)(X)
=\conv\menge{\Lambda_{\ad}y}{y\in\overline{\partial}\varphi\brk!{\gamma(X)}
\,\,\text{and}\,\,\ad\in\AD_X}.
\end{equation}
\end{proposition}
\begin{proof}
Denote by $D$ the set of all $y\in\XX$ for which there exists a
sequence $(x_n)_{n\in\NN}$ in $\XX$ such that
\begin{equation}
\begin{cases}
x_n\to\gamma(X)\\
\text{for every $n\in\NN$,
$\varphi$ is Fr\'echet differentiable at $x_n$}\\
\nabla\varphi(x_n)\to y.
\end{cases}
\end{equation}
In the light of \cref{l:2} (applied to $\Omega=\emp$),
$\overline{\partial}\varphi\brk{\gamma(X)}=\conv D\neq\emp$,
which leads to
\begin{equation}
\label{e:xt43}
\conv\menge{\Lambda_{\bd}v}{v\in D\,\,\text{and}\,\,\bd\in\AD_X}
=\conv\menge{\Lambda_{\bd}v}{v\in\overline{\partial}\varphi\brk!{\gamma(X)}
\,\,\text{and}\,\,\bd\in\AD_X}.
\end{equation}
The nonexpansiveness of $\gamma$ (\cref{p:8}\,\cref{p:8i})
entails that $\varphi\circ\gamma$ is
locally Lipschitz continuous near $X$.
Now let $Y\in\FH$ be a point for which there exists a sequence
$(X_n)_{n\in\NN}$ in $\FH$ such that
\begin{equation}
\label{e:o29z}
\begin{cases}
X_n\to X\\
\text{for every $n\in\NN$, $\varphi\circ\gamma$ is Fr\'echet
differentiable at $X_n$}\\
\nabla(\varphi\circ\gamma)(X_n)\to Y
\end{cases}
\end{equation}
(such a point exists thanks to \cref{l:2}\,\cref{l:2ii}).
Since $\gamma$ is nonexpansive,
\begin{equation}
\label{e:c2su}
\gamma(X_n)\to\gamma(X).
\end{equation}
Next, for every $n\in\NN$, take $\ad_n\in\AD_{X_n}$.
By \cref{c:5}\,\cref{c:5ii},
\begin{equation}
\label{e:c2sv}
(\forall n\in\NN)\quad
\begin{cases}
\text{$\varphi$ is Fr\'echet differentiable at $\gamma(X_n)$}\\
\nabla(\varphi\circ\gamma)(X_n)
=\Lambda_{\ad_n}\brk!{\nabla\varphi\brk!{\gamma(X_n)}}.
\end{cases}
\end{equation}
Hence, because $(\Lambda_{\ad_n})_{n\in\NN}$ are linear isometries,
\begin{equation}
\sup_{n\in\NN}\norm!{\nabla\varphi\brk!{\gamma(X_n)}}
=\sup_{n\in\NN}\norm{\nabla(\varphi\circ\gamma)(X_n)}
<\pinf.
\end{equation}
Therefore, appealing to \cref{p:54}, we obtain a strictly increasing
sequence $(k_n)_{n\in\NN}$ in $\NN$, a point $y\in\XX$, and an element
$\ad\in\AD$ such that
\begin{equation}
\label{e:c2sw}
\nabla\varphi\brk!{\gamma(X_{k_n})}\to y
\quad\text{and}\quad
\Lambda_{\ad_{k_n}}\to\Lambda_{\ad}.
\end{equation}
In turn, we derive from \cref{e:o29z,e:c2sv} that
$Y
=\lim\nabla(\varphi\circ\gamma)(X_{k_n})
=\lim\Lambda_{\ad_{k_n}}\brk{\nabla\varphi\brk{\gamma(X_{k_n})}}
=\Lambda_{\ad}y$.
At the same time, combining \cref{e:c2su,e:c2sv,e:c2sw} yields
$y\in D$, while \cref{e:o29z,p:5} ensure that $\ad\in\AD_X$.
Thus, we deduce from \cref{l:2}\,\cref{l:2ii}
and \cref{e:xt43} that
\begin{align}
\overline{\partial}(\varphi\circ\gamma)(X)
&\subset\conv\menge{\Lambda_{\bd}v}{v\in D\,\,\text{and}\,\,\bd\in\AD_X}
\nonumber\\
&=\conv\menge{\Lambda_{\bd}v}{v\in\overline{\partial}\varphi\brk!{\gamma(X)}
\,\,\text{and}\,\,\bd\in\AD_X}.
\end{align}
We now establish the converse inclusion.
Toward this end, take $v\in D$ and $\bd\in\AD_X$, and
let $(z_n)_{n\in\NN}$ be a sequence in $\XX$ such that
\begin{equation}
\begin{cases}
z_n\to\gamma(X)\\
\text{for every $n\in\NN$,
$\varphi$ is Fr\'echet differentiable at $z_n$}\\
\nabla\varphi(z_n)\to v.
\end{cases}
\end{equation}
Then $\Lambda_{\bd}z_n\to\Lambda_{\bd}\gamma(X)=X$
and $\Lambda_{\bd}(\nabla\varphi(z_n))\to\Lambda_{\bd}v$.
On the other hand, for every $n\in\NN$, \cref{c:5}\,\cref{c:5i}
asserts that $\varphi\circ\gamma$ is Fr\'echet
differentiable at $\Lambda_{\bd}z_n$ with
$\nabla(\varphi\circ\gamma)(\Lambda_{\bd}z_n)
=\Lambda_{\bd}(\nabla\varphi(z_n))$.
Thus, \cref{l:2}\,\cref{l:2ii} implies that
$\Lambda_{\bd}v\in\overline{\partial}(\varphi\circ\gamma)(X)$, and
we conclude the proof by invoking \cref{e:xt43} and the convexity of
$\overline{\partial}(\varphi\circ\gamma)(X)$.
\end{proof}

\begin{remark}
In \cref{p:3}, it is not clear whether the convex hull operation
on the right-hand side of \cref{e:udqz} can be omitted
in general, even though this is indeed the case for the particular
instances of
spectral decomposition systems in \cref{ex:3}, \cref{ex:4},
\cref{ex:5} (with $\KB\in\set{\RR,\CC}$)
as was established in
\cite[Theorem~21(ii)]{Lourenco20},
\cite[Theorem~8]{Lewis99b} (see also \cite[Theorem~1.4]{Lewis96c}),
and \cite[Theorems~3.7 and 4.2]{Lewis05a},
respectively.
\end{remark}

\section{A generalization of Lidski\u{\i}'s theorem}
\label{sec:5}

A pivotal result in the perturbation analysis of Hermitian matrices is
Lidski\u{\i}'s theorem \cite{Lidskii50},
which quantifies how eigenvalues vary under additive perturbations.
More precisely, using the notation of \cref{ex:4}, it asserts that
\begin{equation}
  \label{e:lidskii}
  \brk!{\forall X\in\HS^N(\CC)}
  \brk!{\forall Y\in\HS^N(\CC)}\quad
  \lambda(X+Y)-\lambda(X)\in\conv\brk!{\PS^N\cdot\lambda(Y)};
\end{equation}
see also \cite[Corollary~III.4.2 and Theorem~II.1.10]{Bhatia97}.
This type of majorization has been established beyond the Hermitian
setting to encompass a range of algebraic frameworks, including
the Lie-theoretic framework \cite{Berezin56},
Eaton triples \cite[Theorem~6.4]{Hill01},
singular values of rectangular matrices \cite[Theorem~3.4.5]{Horn91},
and eigenvalues in Euclidean Jordan algebras \cite[Theorem~5.1]{Jeong20}.
However, as illustrated earlier, these settings can be viewed as
spectral decomposition systems.
This thus motivates the following generalization
of Lidski\u{\i}'s theorem, which brings together these results under the
umbrella of our framework.

Establishing \cref{t:3} necessitates the following three
preliminary results. We will employ the notation
\begin{equation}
\begin{cases}
(\forall x\in\XX)\,\,
\SD\cdot x=\menge{\sd\cdot x}{\sd\in\SD}\\
(\forall D\in 2^{\XX})(\forall\sd\in\SD)\,\,
\sd\cdot D=\menge{\sd\cdot x}{x\in D}.
\end{cases}
\end{equation}

\begin{proposition}
\label{p:22}
Suppose that \cref{a:1} is in force and, in addition, that
$\SD$ is a finite set. Let $x\in\XX$ and let $y\in\ran\gamma$.
Then $x\in\conv(\SD\cdot y)$ if and only if $(\forall z\in\XX)$
$\scal{x}{z}\leq\scal{y}{\tau(z)}$.
\end{proposition}
\begin{proof}
Set $C=\conv(\SD\cdot y)$.
Then $C$ is compact as the convex hull of the finite set
$\SD\cdot y$ \cite[Corollary~2.30]{Rock09}.
Therefore, since \cite[Proposition~7.11]{Livre1} implies that
\begin{equation}
x\in C
\quad\Leftrightarrow\quad
(\forall z\in\XX)\,\,
\scal{x}{z}\leq\sup_{w\in C}\scal{w}{z},
\end{equation}
it suffices to verify that
$(\forall z\in\XX)$
$\sup_{w\in C}\scal{w}{z}=\scal{y}{\tau(z)}$.
To do so,
let $\tau$ be the spectral-induced ordering mapping
of the system $(\XX,\SD,\gamma,(\Lambda_{\ad})_{\ad\in\AD})$
(see property~\cref{d:1a} in \cref{d:1})
and let $z\in\XX$.
Then $\tau(y)=y$ \cite[Proposition~3.5\,(i)]{PartI}
and property~\cref{d:1a} in \cref{d:1} entails
that there exists $\ts\in\SD$ such that
$z=\ts\cdot\tau(z)$.
In turn, it results from \cref{p:7} that
$\scal{\ts\cdot y}{z}
=\scal{\ts\cdot y}{\ts\cdot\tau(z)}
=\scal{y}{\tau(z)}
=\scal{\tau(y)}{\tau(z)}$.
Hence, we derive from
\cite[Proposition~11.1\,(iii)]{Livre1}
and \cite[Proposition~3.3\,(iii)]{PartI}
that
$\sup_{w\in C}\scal{w}{z}
=\sup_{\sd\in\SD}\scal{\sd\cdot y}{z}
\leq\scal{\tau(y)}{\tau(z)}
=\scal{\ts\cdot y}{z}
\leq\sup_{\sd\in\SD}\scal{\sd\cdot y}{z}$,
which completes the proof.
\end{proof}

\begin{proposition}
\label{p:58}
Suppose that \cref{a:1} is in force.
Let $x\in\XX$ and $y\in\XX$ be such that
$x\in\conv(\SD\cdot y)$.
Then $\tau(x)\in\conv(\SD\cdot\tau(y))$.
\end{proposition}
\begin{proof}
Set $(\forall\sd\in\SD)$ $\pi_{\sd}\colon\XX\to\XX\colon
z\mapsto\sd\cdot z$.
Then $(\XX,\SD,\tau,(\pi_{\sd})_{\sd\in\SD})$ is a spectral
decomposition system for $\XX$ \cite[Example~3.4]{PartI}.
Now let $(\sd_i)_{i\in I}$ and $(\alpha_i)_{i\in I}$
be finite families in $\SD$ and $\rzeroun$, respectively,
such that $\sum_{i\in I}\alpha_i=1$ and
$x=\sum_{i\in I}\alpha_i\sd_i\cdot y$.
The $\SD$-invariance of $\tau$ entails that
$(\forall i\in I)$ $\tau(\sd_i\cdot y)=\tau(y)$.
Therefore, applying \cite[Proposition~4.11]{PartI} to the spectral
decomposition system $(\XX,\SD,\tau,(\pi_{\sd})_{\sd\in\SD})$,
we obtain
\begin{equation}
\tau(x)
=\tau\brk3{\sum_{i\in I}\alpha_i\sd_i\cdot y}
\in\conv\brk3{\SD\cdot\sum_{i\in I}\alpha_i\tau(\sd_i\cdot y)}
=\conv\brk!{\SD\cdot\tau(y)},
\end{equation}
as announced.
\end{proof}

The last preliminary result is motivated by
\cite[Lemma~2.1]{Niezgoda98}.

\begin{proposition}
\label{p:23}
Suppose that \cref{a:1} is in force.
Let $\tau$ be the spectral-induced ordering mapping
of the system $(\XX,\SD,\gamma,(\Lambda_{\ad})_{\ad\in\AD})$
(see property~\cref{d:1a} in \cref{d:1}),
set $K=\ran\tau$,
and let $\sd$ and $\ts$ be in $\SD$.
Suppose that $(\sd\cdot\inte K)\cap(\ts\cdot K)\neq\emp$.
Then $(\forall x\in K)$ $\sd\cdot x=\ts\cdot x$.
\end{proposition}
\begin{proof}
Take $x\in K$, and let $y\in\inte K$ and $z\in K$ be such that
$\sd\cdot y=\ts\cdot z$. Then there exists $\varepsilon\in\zeroun$
such that, for every $\xi\in\intv{-\varepsilon}{\varepsilon}$,
$y+\xi(x-y)\in K$.
Set $(\forall\xi\in\intv{-\varepsilon}{\varepsilon})$
$y_{\xi}=y+\xi(x-y)$.
Since $\tau$ is idempotent \cite[Proposition~3.3\,(i)]{PartI} and
$\SD$-invariant, we get
\begin{equation}
\begin{cases}
\tau(\sd\cdot x)=\tau(x)=x\\
y=\tau(y)=\tau(\sd\cdot y)=\tau(\ts\cdot z)=\tau(z)=z\\
\brk!{\forall\xi\in\intv{-\varepsilon}{\varepsilon}}\,\,
\tau(\ts\cdot y_{\xi})=\tau(y_\xi)=y_{\xi}.
\end{cases}
\end{equation}
Hence $\sd\cdot y=\ts\cdot z=\ts\cdot y$.
In turn, it follows from \cref{p:7} that
$\scal{\sd\cdot x}{\ts\cdot y}
=\scal{\sd\cdot x}{\sd\cdot y}
=\scal{x}{y}$.
Thus, we derive from \cite[Proposition~3.3\,(iii)]{PartI} that
\begin{align}
\brk!{\forall\xi\in\intv{-\varepsilon}{\varepsilon}}\quad
\xi\scal{\sd\cdot x}{\ts\cdot(x-y)}
&=\scal{\sd\cdot x}{\ts\cdot y+\xi\ts\cdot(x-y)}
-\scal{\sd\cdot x}{\ts\cdot y}
\nonumber\\
&=\scal{\sd\cdot x}{\ts\cdot y_{\xi}}
-\scal{x}{y}
\nonumber\\
&\leq\scal{\tau(\sd\cdot x)}{\tau(\ts\cdot y_{\xi})}
-\scal{x}{y}
\nonumber\\
&=\scal{x}{y_{\xi}}-\scal{x}{y}
\nonumber\\
&=\xi\scal{x}{x-y},
\end{align}
which implies that
$\scal{\sd\cdot x}{\ts\cdot(x-y)}=\scal{x}{x-y}$.
Therefore
$\norm{\sd\cdot x-\ts\cdot x}^2
=\norm{\sd\cdot x}^2-2\scal{\sd\cdot x}{\ts\cdot x}
+\norm{\ts\cdot x}^2
=2\norm{x}^2-2\scal{x}{x-y}-2\scal{\sd\cdot x}{\ts\cdot y}
=2\norm{x}^2-2\scal{x}{x-y}-2\scal{x}{y}
=0$.
Consequently $\sd\cdot x=\ts\cdot x$.
\end{proof}

We can now state and prove the announced generalization of
Lidski\u{\i}'s theorem.

\begin{theorem}
\label{t:3}
Suppose that \cref{a:1} is in force and, in addition, that
$\SD$ is a finite set. Then
\begin{equation}
(\forall X\in\FH)(\forall Y\in\FH)\quad
\gamma(X+Y)-\gamma(X)\in\conv\brk!{\SD\cdot\gamma(Y)}.
\end{equation}
\end{theorem}
\begin{proof}
Let $\tau$ be the spectral-induced ordering mapping
of the system $(\XX,\SD,\gamma,(\Lambda_{\ad})_{\ad\in\AD})$,
and set
\begin{equation}
K=\ran\tau
\quad\text{and}\quad
(\forall\sd\in\SD)\,\,
D_{\sd}=\inte(\sd\cdot K).
\end{equation}
Then $K$ is a closed convex cone in $\XX$
\cite[Proposition~3.5\,(ii)]{PartI} and, thanks to \cref{e:da2t},
\begin{equation}
\label{e:iym1}
\XX=\bigcup_{\sd\in\SD}\sd\cdot K.
\end{equation}
On the other hand, for every $\sd\in\SD$,
since $x\mapsto\sd\cdot x$ is a linear isometry, we get
\begin{equation}
\label{e:142a}
D_{\sd}=\inte(\sd\cdot K)=\sd\cdot\inte K.
\end{equation}
Hence, applying Ursescu's lemma \cite[Lemma~1.44\,(i)]{Livre1} to
the finite family $(\sd\cdot K)_{\sd\in\SD}$ of closed subsets of
$\XX$, we obtain
$\overline{\bigcup_{\sd\in\SD}\sd\cdot\inte K}
=\overline{\bigcup_{\sd\in\SD}\inte(\sd\cdot K)}
=\overline{\inte\bigcup_{\sd\in\SD}\sd\cdot K}
=\XX$.
Therefore $\inte K\neq\emp$.
In turn, we learn from \cref{p:23} that
\begin{equation}
\label{e:iym2}
(\forall\sd\in\SD)(\forall\ts\in\SD)\quad
(\sd\cdot\inte K)\cap(\ts\cdot K)\neq\emp
\quad\Rightarrow\quad
\brk[s]!{\,(\forall x\in K)\,\,
\sd\cdot x=\ts\cdot x
\,}.
\end{equation}
This and \cref{e:iym1} lead to
\begin{equation}
\XX\smallsetminus\bigcup_{\sd\in\SD}D_{\sd}
=\XX\smallsetminus\bigcup_{\sd\in\SD}\inte(\sd\cdot K)
=\bigcap_{\sd\in\SD}\brk!{\XX\smallsetminus\inte(\sd\cdot K)}
=\bigcup_{\sd\in\SD}\bdry(\sd\cdot K),
\end{equation}
where $\bdry D$ denotes the boundary of a subset $D$ of $\XX$.
At the same time, for every $\sd\in\SD$,
since $\sd\cdot K$ is convex, $\bdry(\sd\cdot K)$ is a Lebesgue
null set \cite[Theorem~1]{Lang86}.
Thus, since $\SD$ is finite,
\begin{equation}
\label{e:7lpr}
\XX\smallsetminus\bigcup_{\sd\in\SD}D_{\sd}
\,\,\text{is a Lebesgue null set}.
\end{equation}
Now let $X$ and $Y$ be in $\FH$.
In the light of \cref{p:22}, the desired assertion
is equivalent to
\begin{equation}
\label{e:t0d0}
(\forall z\in\XX)\quad
\scal{\gamma(X+Y)-\gamma(X)}{z}
\leq\scal{\gamma(Y)}{\tau(z)}.
\end{equation}
To prove this, take $z\in\XX$ and define an $\SD$-invariant function
$\varphi\colon\XX\to\RR$ via
\begin{equation}
\label{e:k05w}
(\forall x\in\XX)\quad
\varphi(x)=\scal{\tau(x)}{z}.
\end{equation}
The nonexpansiveness of $\tau$ \cite[Proposition~3.3\,(iv)]{PartI}
implies that $\varphi$ is Lipschitz continuous on $\XX$.
Therefore, since $\gamma$ is nonexpansive, $\varphi\circ\gamma$ is
Lipschitz continuous on $\FH$.
To proceed further, define $D=\bigcup_{\sd\in\SD}D_{\sd}$.
We claim that
\begin{equation}
\label{e:22j0}
\text{$\varphi$ is Fr\'echet differentiable on
$D$ and $(\forall x\in D)$
$\nabla\varphi(x)\in\SD\cdot z$}.
\end{equation}
Toward this end, fix temporarily $\sd\in\SD$ and $x\in D_{\sd}$.
Property~\cref{d:1a} in \cref{d:1} states that there exists
$\ts\in\SD$ such that $x=\ts\cdot\tau(x)$. In turn,
it results from \cref{e:142a} that
$x\in(\sd\cdot\inte K)\cap(\ts\cdot K)$,
from which and \cref{e:iym2} we obtain
$\sd\cdot\tau(x)=\ts\cdot\tau(x)=x$.
Thus, thanks to \cref{p:7},
$\varphi(x)
=\scal{\tau(x)}{z}
=\scal{\sd\cdot\tau(x)}{\sd\cdot z}
=\scal{x}{\sd\cdot z}$.
Therefore, we have shown that
\begin{equation}
\label{e:ph02}
(\forall\sd\in\SD)(\forall x\in D_{\sd})\quad
\varphi(x)=\scal{x}{\sd\cdot z},
\end{equation}
which implies \cref{e:22j0}.
Thus, since $\conv(\SD\cdot z)$ is compact as the convex
hull of the finite set $\SD\cdot z$,
we deduce from \cref{e:7lpr}, \cref{e:22j0}, and
\cref{l:2}\,\cref{l:2ii} that
\begin{equation}
\label{e:conv}
(\forall x\in\XX)\quad
\overline{\partial}\varphi(x)\subset\conv(\SD\cdot z).
\end{equation}
Next, since $\varphi\circ\gamma$ is Lipschitz continuous,
Lebourg's mean value theorem \cite[Theorem~2.3.7]{Clarke90}
(applied to $X+Y$ and $X$)
guarantees that existence of
$V\in\menge{X+\alpha Y}{\alpha\in\intv{0}{1}}$ and
$W\in\overline{\partial}(\varphi\circ\gamma)(V)$, for which
\begin{equation}
\label{e:276a}
\scal{\gamma(X+Y)-\gamma(X)}{z}
=(\varphi\circ\gamma)(X+Y)-(\varphi\circ\gamma)(X)
=\scal{W}{Y},
\end{equation}
where the first identity follows from
\cite[Proposition~3.5\,(i)]{PartI}.
In turn, \cref{p:3} asserts that there exist finite families
$(w_i)_{i\in I}$ in $\overline{\partial}\varphi(\gamma(V))$,
$(\ad_i)_{i\in I}$ in $\AD_V$,
and $(\alpha_i)_{i\in I}$ in $\rzeroun$ such that
\begin{equation}
\label{e:276b}
\sum_{i\in I}\alpha_i=1
\quad\text{and}\quad
W=\sum_{i\in I}\alpha_i\Lambda_{\ad_i}w_i.
\end{equation}
Appealing to \cref{p:58}, we infer from 
\cref{e:conv} that
$\set{\tau(w_i)}_{i\in I}\subset\conv(\SD\cdot\tau(z))$.
Thus, since $\tau(z)\in\ran\tau=\ran\gamma$
\cite[Proposition~3.5\,(ii)]{PartI},
it follows from \cref{p:22} and
\cite[Proposition~3.5\,(i)]{PartI} that
$(\forall i\in I)$
$\scal{\tau(w_i)}{\gamma(Y)}
\leq\scal{\tau(z)}{\tau\brk{\gamma(Y)}}
=\scal{\tau(z)}{\gamma(Y)}$.
Therefore, using \cref{e:276a}, \cref{e:276b}, and
properties~\cref{d:1c} and \cref{d:1a} in
\cref{d:1}, we get
\begin{align}
\scal{\gamma(X+Y)-\gamma(X)}{z}
&=\sum_{i\in I}\alpha_i\scal{\Lambda_{\ad_i}w_i}{Y}
\nonumber\\
&\leq\sum_{i\in I}\alpha_i\scal{\gamma(\Lambda_{\ad_i}w_i)}{\gamma(Y)}
\nonumber\\
&=\sum_{i\in I}\alpha_i\scal{\tau(w_i)}{\gamma(Y)}
\nonumber\\
&\leq\sum_{i\in I}\alpha_i\scal{\tau(z)}{\gamma(Y)}
\nonumber\\
&=\scal{\gamma(Y)}{\tau(z)},
\end{align}
which completes the proof.
\end{proof}

\begin{remark}
\label{r:3}
Here are several noteworthy instantiations of \cref{t:3}.
\begin{enumerate}
\item In the context of Hermitian matrices of \cref{ex:4},
\cref{t:3} reads
\begin{equation}
\brk!{\forall X\in\HS^N(\KB)}
\brk!{\forall Y\in\HS^N(\KB)}\quad
\lambda(X+Y)-\lambda(X)\in\conv\brk!{\PS^N\cdot\lambda(Y)},
\end{equation}
and we thus recover Lidski\u{\i}'s theorem \cite{Lidskii50}.
\item In the Euclidean Jordan algebra framework of
\cref{ex:3}, \cref{t:3} reduces to
\cite[Theorem~5.1]{Jeong20}, that is,
\begin{equation}
(\forall X\in\FH)(\forall Y\in\FH)\quad
\lambda(X+Y)-\lambda(X)\in\conv\brk!{\PS^N\cdot\lambda(Y)}.
\end{equation}
\item By specializing \cref{t:3} to \cref{ex:5} with
$\KB=\CC$, we obtain the version of Lidski\u{\i}'s theorem for
singular values \cite[Theorem~IV.3.4 and Exercise~II.2.10]{Bhatia97}, that is,
\begin{equation}
\brk!{\forall X\in\CC^{M\times N}}
\brk!{\forall Y\in\CC^{M\times N}}
\quad
\sigma(X+Y)-\sigma(X)\in\conv\brk!{\PS_{\pm}^m\cdot\sigma(Y)}.
\end{equation}
\end{enumerate}
\end{remark}

\section{Conclusion}

Continuing from our previous work \cite{PartI},
we have developed a unified approach to generalized-subdifferential
formulae for spectral functions in the framework of spectral
decomposition systems.
This framework covers a wide range of previously studied settings,
including eigenvalue or singular value decompositions of real,
complex, and quaternion matrices, as well as Euclidean Jordan
algebras. Specifically, we have derived
representations of regular and limiting normal cones to spectral
sets and, based on these, regular and limiting subdifferentials of
spectral functions. These results further yielded
characterizations of Fréchet differentiability and the
corresponding derivatives, together with a characterization of
Clarke gradients. Using the latter, we
generalized Lidski\u{\i}'s theorem
on the spectrum of additive perturbations of Hermitian matrices to
arbitrary spectral decomposition systems.

This work can be extended in a number of directions. First, the
representations of generalized subdifferentials can be used to
obtain explicit necessary optimality conditions for concrete matrix
optimization problems such as low-rank matrix completion via
nonconvex Schatten $p$-norm penalization in a suitably general
setting. In the context of variational analysis, an important open
question concerns characterizations of Lipschitz-like properties of
solution mappings such as metric regularity or subregularity, cf.
\cite[Chapter 27]{ClasonValkonen} and the literature cited therein.
Such results could then be used to generalize results on
second-order variational analysis of spectral functions such as
\cite{Sarabi25}.


\bibliographystyle{jnsao}
\bibliography{bibli}

\end{document}

%% file: macros.tex
\usepackage[extdef=true]{delimset}
\DeclareMathDelimiterSet{\scal}[2]{
  \selectdelim[l]< {#1} \mathpunct{}\selectdelim[p]| {#2} \selectdelim[r]>
}

\newcommand{\menge}[2]{
  \bigl\{{#1}\mid{#2}\bigr\}
}
\DeclareMathDelimiterSet{\Menge}[2]{
  \selectdelim[l]\{{#1}\selectdelim[m]|{#2}\selectdelim[r]\}
}
\newcommand*\Cdot{{\mkern 1.6mu\cdot\mkern 1.6mu}}
\DeclareMathDelimiterSet{\killing}[2]{
  \selectdelim[l]( {#1} \mathpunct{}\selectdelim[p]| {#2} \selectdelim[r])
}

\newcommand{\AD}{\mathscr{A}}
\newcommand{\DD}{\mathscr{D}}

\newcommand{\ad}{\mathscr{a}}
\newcommand{\bd}{\mathscr{b}}

\newcommand{\CC}{\mathbb{C}}
\newcommand{\KB}{\mathbb{K}}
\newcommand{\HB}{\mathbb{H}}
\newcommand{\NN}{\mathbb{N}}
\newcommand{\RR}{\mathbb{R}}

\newcommand{\HH}{\mathcal{H}}
\newcommand{\GG}{\mathcal{G}}
\newcommand{\UU}{\mathcal{U}}

\newcommand{\XX}{\mathcal{X}}

\newcommand{\HS}{\mathsf{H}}
\newcommand{\US}{\mathsf{U}}
\newcommand{\PS}{\mathsf{P}}
\newcommand{\SD}{\mathsf{S}}
\newcommand{\SO}{\mathsf{SO}}

\newcommand{\ts}{\mathsf{t}}
\newcommand{\sd}{\mathsf{s}}

\newcommand{\FH}{\mathfrak{H}}

\newcommand{\pinf}{{+}\infty}
\newcommand{\minf}{{-}\infty}
\newcommand{\emp}{\varnothing}
\newcommand{\exi}{\exists\,}

\renewcommand{\leq}{\leqslant}
\renewcommand{\geq}{\geqslant}
\newcommand\restr[2]{{
  \left.\kern-\nulldelimiterspace 
    #1 
  \right|_{#2} 
  }
}

\newcommand{\zeroun}{\intv[o]{0}{1}}
\newcommand{\rzeroun}{\intv[l]{0}{1}}

\newcommand{\RXX}{\overline{\RR}}
\newcommand{\RX}{\intv[l]{\minf}{\pinf}}
\newcommand{\RP}{\RR_+}
\newcommand{\RPP}{\intv[o]{0}{\pinf}}
\newcommand{\RPX}{\intv{0}{\pinf}}

\DeclareMathOperator{\ran}{range}

\DeclareMathOperator{\inte}{int}
\DeclareMathOperator{\bdry}{bdry}

\DeclareMathOperator{\sign}{sign}

\newcommand{\Sum}{\displaystyle\sum}
\newcommand{\Id}{\mathrm{Id}}

\DeclareMathOperator{\cone}{cone}
\DeclareMathOperator{\conv}{conv}

\DeclareMathOperator{\epi}{epi}

\DeclareMathOperator{\gra}{gra}

\DeclareMathOperator{\Proj}{Proj}

\newcommand*{\tran}{^{\mkern-1.6mu\mathsf{T}}}
\renewcommand{\Re}{\operatorname{Re}}
\DeclareMathOperator{\tra}{trace}
\DeclareMathOperator{\Tra}{Trace}

\DeclareMathOperator{\Diag}{Diag}